\documentclass[12pt,a4paper,leqno]{article}
\usepackage[latin1]{inputenc}
\usepackage[T1]{fontenc}
\usepackage{lmodern}

\usepackage[intlimits]{amsmath}
\usepackage{amssymb}
\usepackage{amsthm}
\usepackage{bm}
\usepackage{graphicx}

\theoremstyle{plain}
\newtheorem{conjecture}{Conjecture}[section]
\newtheorem{theorem}[conjecture]{Theorem}
\newtheorem{lemma}[conjecture]{Lemma}
\newtheorem{proposition}[conjecture]{Proposition}
\newtheorem{corollary}[conjecture]{Corollary}

\theoremstyle{definition}
\newtheorem{definition}[conjecture]{Definition}
\newtheorem{remark}[conjecture]{Remark}

\newtheorem{example}[conjecture]{Example}
\newtheorem{examples}[conjecture]{Examples}

\numberwithin{equation}{section}

\DeclareMathOperator{\qr}{qr}

\DeclareMathOperator{\emb}{emb}
\DeclareMathOperator{\sub}{sub}

\DeclareMathOperator{\cf}{cf}
\DeclareMathOperator{\cl}{cl}
\DeclareMathOperator{\Str}{Str}

\DeclareMathOperator{\ar}{ar}

\DeclareMathOperator{\frvar}{frvar}

\DeclareMathOperator{\disjoint}{disjoint}

\title{On logics extended with embedding-closed quantifiers}
\date{}
\author{J. Haigora \and K. Luosto}

\begin{document}

\maketitle

\begin{abstract}
We study first-order as well as infinitary logics
extended with quantifiers closed upwards under embeddings.
In particular, we show that if a chain of quasi-homogeneous structures is sufficiently long then, in that chain, a given formula of such a logic becomes eventually equivalent to a quantifier-free formula.
We use this fact to produce a number of undefinability results for logics with embedding-closed quantifiers.
In the final section we introduce an Ehrenfeucht-Fra\"iss\'e game that characterizes the $\mathcal{L}_{\infty \omega}(\mathcal{Q}_{\emb})$-equivalence between structures,
where $\mathcal{Q}_{\emb}$ is the class of all embedding-closed quantifiers. In conclusion, we provide an application of this game illustrating its use.
\end{abstract}

\section{Introduction}

In this paper we focus our attention on a certain class of logics whose expressive power is greater than that of first-order logic (denoted here by $\mathcal{L}_{\omega\omega}$), the central logic in mathematics.
While $\mathcal{L}_{\omega\omega}$ has well-developed model theory due to its many convenient properties, it has a downside in that its expressive power is rather limited. 
Many natural mathematical statements, for example "there are infinitely many", cannot be expressed in $\mathcal{L}_{\omega\omega}$.
This motivates study of alternative logics.
 
Mostowski was one of the first to suggest in \cite{Mostowski: 1957} 
the idea of expanding $\mathcal{L}_{\omega\omega}$ with formulas of the form $Q_{\alpha} x \varphi(x)$ which are interpreted so that $\mathfrak{A} \vDash  Q_{\alpha} x \varphi(x)$ if and only if there are at least $\aleph_{\alpha}$ elements $a$ with $\mathfrak{A} \vDash \varphi(a)$.
This idea broadened the notion of quantifier giving rise to many interesting logics defined in a similar way. 
The current definition of generalized quantifier is due to Lindstr\"om \cite{Lindstrom: 1966}. We describe it in more detail in Section \ref{prel}.
 
In short, every generalized quantifier $Q$ corresponds to some property of structures that we denote by $\mathcal{K}_Q$.
Suppose $\mathcal{L}$ is a logic closed under substitution and $P$ is a property not expressible in it. By adding quantifier $Q_P$ to $\mathcal{L}$ we get the smallest extension of $\mathcal{L}$ satisfying certain closure conditions that can express $P$. 
The properties of the new logic $\mathcal{L}(Q_P)$ can differ substantially from those of $\mathcal{L}$ and may thus become an interesting object of study. 
A justification for studying generalized quantifiers comes among others from the fact that any reasonable extension $\mathcal{L}$ of, say, the logic $\mathcal{L}_{\omega\omega}$,
closed under substitutions, is equivalent to the logic $\mathcal{L}_{\omega\omega}(\mathcal{Q})$ where $\mathcal{Q}$ is a class of quantifiers corresponding to some new properties definable in $\mathcal{L}$.

In the present work
we shall concentrate on extensions of logics $\mathcal{L}_{\omega\omega}$, $\mathcal{L}_{\infty \omega}$ and $\mathcal{L}_{\infty \omega}^{\omega}$ (the finite variable logic) with generalized quantifiers $Q$ that satisfy the following restriction: for all structures $\mathfrak{A} \in \Str[\tau_Q]$, if $\mathfrak{A} \in \mathcal{K}_Q$ and $\mathfrak{A}$ is embeddable into $\mathfrak{B}$ then $\mathfrak{B} \in \mathcal{K}_Q$. 
We call such quantifiers \emph{embedding-closed} and denote the class of all embedding-closed quantifiers by $\mathcal{Q}_{\emb}$.
This is an interesting class of quantifiers to study since many well-known quantifiers and properties, like cardinality quantifiers $Q_{\alpha}$ for instance, are closed under embeddings either upwards or downwards which is essentially the same for our purposes. We present more examples of such properties in Section \ref{emb-cl_quant}.
We also note that embedding-closed quantifiers are natural in the sense that if $Q$ is embedding-closed then  a sentence $Q(\overline{x}_{\alpha}\varphi_{\alpha})_{\alpha<\kappa}$ says that formulas $\varphi_{\alpha}$, $\alpha<\kappa$, define a structure that have a substructure belonging to a given class of structures.
Thus, embedding-closed quantifiers seem to be an interesting object of study, and our observations in this paper show that it is possible to develop general theory for this class of quantifiers.

Call a structure $\mathfrak{A}$ \emph{quasi-homogeneous} if every isomorphism between fi\-ni\-te\-ly generated substructures of $\mathfrak{A}$ can be extended to an embedding of $\mathfrak{A}$ into itself.
This weakens the usual notion of homogeneity which deals with automorphisms instead of embeddings. 
The notion of embedding-closed quantifier arises naturally when we observe in Theorem \ref{homog} that in order to guarantee that a quasi-homogeneous structure has quantifier elimination for logic $\mathcal{L}_{\infty\omega}$ extended with a set of quantifiers $\mathcal{Q}$, the quantifiers in $\mathcal{Q}$ must be closed under embeddings. 

In \cite{Dawar: 2010}, Dawar and Gr\"adel showed that $\mathcal{L}_{\infty \omega}^{\omega}$ augmented with finitely many embedding-closed quantifiers of finite width has a $0$-$1$ law meaning that on finite structures such a logic can only express properties that hold in almost all finite structures.
Our aim is to study further limits of the expressive power of logics with embedding-closed quantifiers that are not implied by a $0$-$1$ law.
These include for example indefinability of properties of infinite structures and structures with function symbols.
In this article we provide two methods that make it possible. The first method involves construction of a certain chain of quasi-homogeneous structures. 
The second method is based     
on the Ehrenfeucht-Fra\"iss\'e game that we develop in order to characterize $\mathcal{L}_{\infty \omega}(\mathcal{Q}_{\emb})$-equivalence between structures. 
In the article we apply these methods to produce a number of undefinability results.

The article is structured as follows.
In Section \ref{prel} we introduce preliminary notions.
In Section \ref{emb-cl_quant} we describe basic properties of embedding-closed quantifiers that will be needed later, and give some examples.
Before moving to our own major results, we show that $\mathcal{L}^{\omega}_{\infty \omega}(\mathcal{Q})$, where $\mathcal{Q}$ is a finite set of embedding-closed quantifiers of finite width, has a $0$-$1$ law. We do this in Section \ref{law}. The proof concerning $0$-$1$ law was originally given in \cite{Dawar: 2010}.
In Section \ref{chain_section} we introduce the notion of quasi-homogeneity and show that if a chain of quasi-homogeneous structures is sufficiently long then the truth value of a given sentence of a logic with embedding-closed quantifiers is eventually preserved.
This in turn allows us to obtain some undefinability results. 
The section has two subsections, one of which is devoted to the undefinability of properties of finite structures and another deals with infinite structures. 
In Section \ref{game_section} we describe the \emph{embedding game} that characterizes $\mathcal{L}_{\infty\omega}(\mathcal{Q}_{\emb})$-equivalence of a given pair of structures. We close the section with an application of the game that allows us to show that
for each $n<\omega$ there is a first-order sentence of quantifier rank $n$ that is not expressible by any sentence of $\mathcal{L}_{\infty\omega}(\mathcal{Q}_{\emb})$ of quantifier rank $<n$.

\section{Preliminaries}\label{prel}
The notation for the basic concepts of abstract logics that we use in this paper is taken from the book \cite{Ebbinghaus: 1985}, and we assume that the reader is familiar with them.
A \emph{vocabulary} $\tau$ consists of \emph{relation}, \emph{function} and \emph{constant symbols},
\[
	\tau = \{R,\dots,f,\dots,c,\dots\}.
\]
We denote by $\ar(R)$ and $\ar(f)$ the \emph{arities} of relation and function symbols, respectively.
A \emph{$\tau$-structure} $\mathfrak{A}$ is a sequence
\[
	\mathfrak{A} = (A,R^{\mathfrak{A}},\dots,f^{\mathfrak{A}},\dots,c^{\mathfrak{A}}),
\]
where $A$ is a set that we call the \emph{universe} of $\mathfrak{A}$,
and $R^{\mathfrak{A}}\subseteq A^{\ar(R)},\dots$ are interpretations of symbols of $\tau$.
We denote the class of all $\tau$-structures by $\Str[\tau]$.
We denote the number of free variables of $\varphi$ by $\frvar(\varphi)$.
If $\mathfrak{A}\in\Str[\tau]$ and $\varphi \in \mathcal{L}[\tau]$ then we write
\[
	\varphi^{\mathfrak{A}} = \{\overline{a} \in A^{\frvar(\varphi)} : \mathfrak{A},\overline{a} \vDash \varphi  \}.
\]
A \emph{literal} is an atomic formula or a negation of an atomic formula.
An \emph{atomic $n$-type} of $\tau$ is a set $\Phi$ of literals of $\tau$ in variables $x_1,\dots,x_n$ such that there is a $\tau$-structure $\mathfrak{A}$ and $n$-tuple $\overline{a}$ of elements in $A$ with
\[
	\Phi = \{\varphi \colon \mathfrak{A},\overline{a} \vDash \varphi \text{ and } \varphi \text{ is a literal of }\tau \}.
\]
If we work with a logic $\mathcal{L}$ in which for each atomic type $\Phi$ there is a formula $t$ equivalent to $\bigwedge\Phi$ then  $t$ is also called an atomic type.

\begin{lemma}\label{quant}
Let $\vartheta$ be a quantifier-free $\tau$-formula in $n$ free variables.
Then there exists a set $T$ of atomic $n$-types of $\tau$ such that
\[
	\vDash \vartheta \leftrightarrow \bigvee_{\Phi\in T} \bigwedge_{\varphi\in\Phi} \varphi.
\]
\end{lemma}

In this article we will consider the following logics.
We assume that the reader is familiar with the \emph{first-order logic} $\mathcal{L}_{\omega\omega}$ and the related notions.
Let $\kappa$ be a cardinal. The logic $\mathcal{L}_{\kappa\omega}$ is allowed to have conjunctions and disjunctions over sets of formulas of cardinality $< \kappa$. 
The logic $\mathcal{L}_{\infty\omega}$ can have conjunctions and disjunctions over arbitrary sets of formulas. 
Formulas of the logic $\mathcal{L}_{\infty\omega}^{\omega}$ (\emph{finite variable logic}) are exactly those of $\mathcal{L}_{\infty \omega}$ that use at most finite number of variables.

Suppose $\mathcal{L}$ is a logic
and $\tau,\sigma$ are vocabularies where $\sigma$ has only relation symbols. An \emph{$\mathcal{L}$-interpretation} of $\sigma$ in $\tau$ is a sequence $(\Psi,(\psi_R)_{R\in\sigma})$, where each $\psi_R$ is an $\mathcal{L}[\tau]$-formula that has exactly $\ar(R)$ free variables,
and $\Psi$ is a function $\Str[\tau] \rightarrow \Str[\sigma]$ such that
for each $\mathfrak{A}\in\Str[\tau]$, the universe of $\Psi(\mathfrak{A})$ is $A$ and
\begin{equation*}
	R^{\Psi(\mathfrak{A})} = \{\overline{a} \in A^{\ar(R)} \colon \mathfrak{A}\vDash\psi_R(\overline{a}) \}
\end{equation*}
for all $R\in\sigma$.

Let $C \subseteq \Str[\sigma]$ be a class of structures closed under isomorphism.
The logic $\mathcal{L}_{\kappa\omega}(Q_C)$ is the smallest extension of $\mathcal{L}_{\kappa\omega}$ closed under negation, conjunctions and disjunctions of cardinality $<\kappa$ and application of the existential quantifier $\exists$ such that  for all
$\mathcal{L}_{\kappa\omega}(Q_C)$-interpretations $(\Psi,(\psi_R)_{R\in \sigma})$ there is a $\mathcal{L}_{\kappa\omega}(Q_C)[\tau]$-formula $\chi$ such that
\begin{equation}\label{clever}
	\mathfrak{A} \vDash \chi \Leftrightarrow \Psi(\mathfrak{A})\in C
\end{equation}
for all $\mathfrak{A} \in \Str[\tau]$.
In a similar way we define the logics $\mathcal{L}_{\infty \omega}(Q_C)$ and $\mathcal{L}^{\omega}_{\infty\omega}(Q_C)$.
We say that $Q_C$ is the \emph{generalized quantifier} corresponding to the class $C$, and write $Q_C(\overline{x}_R\psi_R)_{R\in\sigma}$ for the formula $\chi$ in \eqref{clever}.
The class $C$ is called the \emph{defining class} of $Q_C$, and $\sigma$ the \emph{vocabulary} of $Q$.
Sometimes we denote a quantifier just by the symbol $Q$, its defining class by $K_Q$, and the vocabulary of $Q$ by $\tau_Q$.
We say that $|\tau_Q|$ is the \emph{width} of $Q$.
In a similar way,
we can also define the extension of a given logic $\mathcal{L}$ with a \emph{class} $\mathcal{Q}$ of quantifiers, instead of just one, which we denote by $\mathcal{L}(\mathcal{Q})$.

\clearpage

\begin{lemma}\label{interpr}
Let $\mathcal{Q}$ be a possibly empty class of quantifiers, $\kappa$ a cardinal and $\mathcal{L} = \mathcal{L}_{\kappa \omega}(\mathcal{Q})$ or $\mathcal{L} = \mathcal{L}^{\omega}_{\infty\omega}(\mathcal{Q})$.
Suppose $(\Psi,(\psi_R)_{R\in \sigma})$ is an $\mathcal{L}$-interpretation of $\sigma$ in $\tau$.
Then for each $\mathcal{L}[\sigma]$-formula $\varphi$ there is an $\mathcal{L}[\tau]$-formula $\varphi^*$ such that
\[
	\mathfrak{A},\overline{a} \vDash \varphi* \Leftrightarrow \Psi(\mathfrak{A}),\overline{a} \vDash \varphi 
\]
for all $\mathfrak{A} \in \Str[\tau]$ and tuples $\overline{a}$ of elements in $A$.
\end{lemma}
\begin{proof}
Replace all atomic subformulas $R(\overline{x})$ of $\varphi$ with formulas $\psi_R(\overline{x})$ to get $\varphi^*$. 
\end{proof}

\begin{lemma}\label{possibly_empty}
Let $\mathcal{Q}_0$ be a possibly empty class of quantifiers, $\kappa$ a cardinal and $\mathcal{L} = \mathcal{L}_{\kappa \omega}(\mathcal{Q}_0)$ or $\mathcal{L} = \mathcal{L}^{\omega}_{\infty\omega}(\mathcal{Q}_0)$.
Suppose $\mathcal{Q}_1$ is a class of quantifiers such that for every $Q\in\mathcal{Q}_1$ its defining class $\mathcal{K}_Q$ is definable in $\mathcal{L}$.
Then $\mathcal{L}(\mathcal{Q}_1) \equiv \mathcal{L}$.
\end{lemma}
\begin{proof}
Let $\tau$ be a vocabulary and consider a formula $Q(\overline{x}_R\varphi_R)_{R\in\tau_Q}$ with $Q \in \mathcal{Q}_1$ and $\varphi_R \in \mathcal{L}[\tau]$ for all $R\in\tau_Q$.
Let $\psi \in \mathcal{L}[\tau_Q]$ be the sentence defining $\mathcal{K}_Q$.
Then 
\[
	\vDash Q(x_R\varphi_R)_{R\in\tau_Q} \leftrightarrow \chi
\]
where $\chi$ is the formula got from $\psi$ by substituting every atomic subformula $R(\overline{x})$ of $\psi$ with $\varphi_R(\overline{x})$.
\end{proof}

\section{Embedding-closed quantifiers}\label{emb-cl_quant}

\begin{definition}
Let $\mathfrak{A}$ and $\mathfrak{B}$ be structures of the same vocabulary $\tau$. An injection $f \colon A \rightarrow B$ is an \emph{embedding of $\mathfrak{A}$ into $\mathfrak{B}$} if
\begin{enumerate}
\item $f(c^{\mathfrak{A}}) = c^{\mathfrak{B}}$ for all constant symbols $c \in \tau$,
\item $\overline{a} \in R^{\mathfrak{A}} \Leftrightarrow f\overline{a} \in R^{\mathfrak{B}}$ for all relation symbols $R \in \tau$ and tuples $\overline{a}$ in $A$,
\item $fF^{\mathfrak{A}}(\overline{a}) = F^{\mathfrak{B}}(f\overline{a})$ for all function symbols $F \in \tau$ and tuples $\overline{a}$ in $A$.
\end{enumerate}
The notation $\mathfrak{A} \leq \mathfrak{B}$ means that $\mathfrak{A}$ is embeddable into $\mathfrak{B}$.
We say that a sequence $(\mathfrak{A}_{\alpha})_{\alpha < \gamma}$ of $\tau$-structures is a \emph{chain} if $\mathfrak{A}_{\alpha} \leq \mathfrak{B}_{\beta}$ whenever $\alpha < \beta$.
We say that a class $C$ of $\tau$-structures is an \emph{antichain} if $\mathfrak{A} < \mathfrak{B}$ is never true for any structures $\mathfrak{A}, \mathfrak{B} \in C$.

A class $K$ of $\tau$-structures is \emph{embedding-closed} if $\mathfrak{A} \in K$ and $\mathfrak{A} \leq \mathfrak{B}$ imply $\mathfrak{B} \in K$. We say that a quantifier $Q$ is embedding-closed if its defining class is embedding-closed.
We denote by $\mathcal{Q}_{\emb}$ the class of all embedding-closed quantifiers.

\end{definition}

\begin{lemma}\label{preservation}
Let $\tau$ be a vocabulary,
$(\varphi_{\alpha} )_{\alpha < \kappa}$ quantifier-free $\tau$-formulas and $Q$ an embedding-closed quantifier of width $\kappa$. 
The formula $Q(\overline{x}_{\alpha} \varphi_{\alpha})_{\alpha < \kappa}$ is preserved by embeddings.
\end{lemma}
\begin{proof}
Let $\mathfrak{A}$ and $\mathfrak{B}$ be $\tau$-structures.
Suppose that $(\mathfrak{A},\overline{a}) \vDash Q(\overline{x}_{\alpha} \varphi_{\alpha})_{\alpha < \kappa}$ and $f \colon A \rightarrow B$ is an embedding.
Then $f$ is also an embedding of $(A,(\varphi_{\alpha}^{\mathfrak{A},\overline{a}})_{\alpha < \kappa})$ into $(B,(\varphi_{\alpha}^{\mathfrak{B},f\overline{a}})_{\alpha < \kappa})$ since quantifier-free formulas are preserved by embeddings, so $(\mathfrak{B},f\overline{a}) \vDash Q(\overline{x}_{\alpha} \varphi_{\alpha})_{\alpha < \kappa}$ since $Q$ is embedding-closed.
\end{proof}

Note that instead of requiring the quantifiers to be closed upwards under embeddings, we could use the downwards closure to get an equivalent class of quantifiers.   
Call a quantifier $Q$ \emph{substructure-closed} if from $\mathfrak{A} \in K_Q$ and $\mathfrak{B} \leq \mathfrak{A}$ follows $\mathfrak{B} \in K_Q$, and denote the class of all substructure-closed quantifiers by $\mathcal{Q}_{\sub}$. The expressive power of $\mathcal{Q}_{\sub}$ is clearly the same as that of $\mathcal{Q}_{\emb}$ since the complement $Q^*$ of an embedding-closed quantifier $Q$ is substructure-closed, so
\[
	\mathfrak{A} \vDash Q(\overline{x}_{\alpha}\varphi_{\alpha})_{\alpha < \kappa} \Leftrightarrow  \mathfrak{A} \vDash \neg Q^*(\overline{x}_{\alpha}\varphi_{\alpha})_{\alpha < \kappa}.
\]

Next we present some examples of well-known properties and quantifiers that are either embedding-closed or are definable in the logic $\mathcal{L}_{\infty \omega}(\mathcal{Q}_{\emb})$.
We use notation $Q^{\cl}$ to denote the closure of the quantifier $Q$ under embeddings. In other words $Q^{\cl}$ is the smallest embedding-closed quantifier containig $Q$.

\begin{examples}
\begin{enumerate}
\item Let $\tau = \{U\}$ be a vocabulary consisting of a single unary relation symbol. The existential quantifier $\exists$ corresponds to the class of structures
$\{\mathfrak{A} \in \Str[\tau] : U^{\mathfrak{A}} \ne \emptyset \}$, so it is embedding-closed.

\item Let $\tau$ be the same as above and $\alpha$ an ordinal. The defining class of the \emph{cardinality quantifier} $Q_{\alpha}$ is $\{\mathfrak{A} \in \Str[\tau] : |U^{\mathfrak{A}}| \geq \aleph_{\alpha} \}$
which is clearly embedding-closed.

\item For each $n < \omega$, let $\sigma_n = \{M_n\}$ be a vocabulary consisting of a single $n$-ary relation symbol.
The \emph{Magidor-Malitz} quantifier $Q_{\alpha}^n$, whose defining class is
\[
	\{ \mathfrak{A} \in \Str[\sigma_n] : \text{ there is } C \subseteq A \text{ with } |C| \geq \aleph_{\alpha} \text{ and } C^n \subseteq M_n^{\mathfrak{A}}\},
\]
is embedding-closed.

\item The well-ordering quantifier $Q^W$, whose defining class is the class of all well-orders,
is substructure-closed, so by our earlier remark it can be defined with embedding-closed quantifiers.

\item The equivalence quantifier $Q^E_{\alpha}$, whose defining class consists of all
structures $(A,E)$ where $E$ is an equivalence relation on $A$ with at least 
$\aleph_{\alpha}$ equivalence classes, is not embedding-closed.
Nonetheless, it can be defined by the sentence
\[
	(Q^E_{\alpha})^{\cl}xyE(x,y) \land \text{"}E \text{ is an equivalence relation"}
\]
in $\mathcal{L}_{\omega\omega}(Q)$ where $Q = (Q^E_{\alpha})^{\cl}$ is embedding-closed.

\item Many graph properties are embedding- or substructure-closed. Examples include $k$-colorability, being a forest, completeness, planarity, having a cycle, and many others.

\item 
This is an example of a graph property that is not embedding- or substructure-closed but is however definable in $\mathcal{L}_{\omega\omega}(Q)$ for an em\-bed\-ding-closed quantifier $Q$.
The property in question is connectedness of a graph. 
Let $\sigma = \{R,B,E\}$ be the vocabulary of
coloured graphs where symbols $R$ and $B$ stand for colors \emph{red} and \emph{blue}.
Let $C \subseteq \Str[\sigma]$ consist of all the graphs $G$ in which for every blue-red pair $(x,y)$ of vertices there is a path between $x$ and $y$ and $R^G$, $B^G$ are not empty. 
Put $D = C^{\cl}$.
Then for all graphs $G$, 
\[
	G \vDash \forall xy Q_Dstuv(s=x,t=y,E(u,v))
\]
if and only if $G$ is connected.
\end{enumerate}
\end{examples}

As we will show below, there exist properties not definable in $\mathcal{L}_{\infty \omega}(\mathcal{Q}_{\emb})$. 
These include among others equicardinality of sets (Example \ref{Hartig}), and completeness and cofinality of an ordering (Examples \ref{completeness} and \ref{cofinality}).

\subsection{Homomorphism-closed quantifiers}

In addition to being closed under embeddings, we can think of various other interesting closure conditions for quantifiers, like being closed under homomorphisms for instance.
The purpose of this subsection is to show that any embedding-closed property can be defined in a logic with homomorphism-closed quantifiers only. Thus, 
requiring quantifiers to be closed under homomorphisms is not essentially more restrictive than the requirement that they are closed under embeddings.
We denote by $\mathcal{Q}_{\hom}$ the class of all homomorphism-closed quantifiers.

\begin{theorem}
Let $\mathcal{L} = \mathcal{L}_{\kappa\omega}$ for some cardinal $\kappa$ or $\mathcal{L} = \mathcal{L}_{\infty\omega}^{\omega}$.
Then $\mathcal{L}(\mathcal{Q}_{\emb}) \equiv \mathcal{L}(\mathcal{Q}_{\hom})$.
\end{theorem}
\begin{proof}
Let $\tau$ be a relational vocabulary and $K$ an embedding-closed class of $\tau$-structures.
Let $\tau' = \tau \cup \{R_* : R \in \tau \} \cup \{N\}$ with $\ar(R_*)= \ar(R)$ for all $R \in\tau$ and $\ar(N) = 2$.
Let $F : \Str[\tau] \rightarrow \Str[\tau']$ be the function such that for each $\mathfrak{A} \in \Str[\tau]$
the universe of $F(\mathfrak{A})$ is $A$, $R^{F(\mathfrak{A})} = R^{\mathfrak{A}}$ and $R_*^{F(\mathfrak{A})} = (\neg R)^{\mathfrak{A}}$ for all $R \in \tau$, and $N^{F(\mathfrak{A})} = (x\ne y)^{\mathfrak{A}}$.  
Now let
\[
	K' = \{\mathfrak{A}\in\tau' : \text{ there is } \mathfrak{B} \in K \text{ and a homomorphism }f:F(\mathfrak{B})\rightarrow \mathfrak{A} \}.
\]
Then $K'$ is homomorphism-closed since the composition of two homomorphisms is itself a homomorphism. 
Let $\psi$ be the following sentence of $\mathcal{L}(\mathcal{Q}_{\hom})$: 
\[
	\psi := Q_{K'}((\overline{x}_RR)_{R\in\tau},(\overline{y}_R\neg R)_{R\in\tau},uv(u\ne v)).
\] 
We claim that $\psi$ defines the class $K$. To show this, suppose first that $\mathfrak{A} \in \Str[\tau]$ is such that $\mathfrak{A}\vDash \psi$.
Then by the definition of $K'$ there is $\mathfrak{B} \in K$ and a homomorphism 
\[
	f : F(\mathfrak{B}) \rightarrow (A,(R^{\mathfrak{A}})_{R\in\tau},(\neg R^{\mathfrak{A}})_{R\in\tau},(x\ne y)^{\mathfrak{A}})
\]
which is in fact an embedding of $\mathfrak{B}$ into $\mathfrak{A}$ so,
since $K$ is embedding-closed, we have $\mathfrak{A} \in K$.
For the other direction, if $\mathfrak{A}\in K$ then $F(\mathfrak{A}) \in K'$ so
\[
	(A,(R^{\mathfrak{A}})_{R\in\tau},(\neg R^{\mathfrak{A}})_{R\in\tau},(x\ne y)^{\mathfrak{A}}) \in K'
\]
from which it follows that $\mathfrak{A} \vDash \psi$.
Thus, $K$ is definable in $\mathcal{L}(\mathcal{Q}_{\hom})$ so by Lemma \ref{possibly_empty} we have $\mathcal{L}(\mathcal{Q}_{\emb}) \leq \mathcal{L}(\mathcal{Q}_{\hom})$, and since every homomomorphism-closed quantifier is also embedding-closed we have  $\mathcal{L}(\mathcal{Q}_{\emb}) \equiv \mathcal{L}(\mathcal{Q}_{\hom})$
\end{proof}

\section{$0$-$1$ law}\label{law}
In \cite{Dawar: 2010}, Dawar and Gr\"adel showed that logic $\mathcal{L}^{\omega}_{\infty \omega}$ (finite variable logic) extended with finitely many embedding-closed quantifiers of finite width has a $0$-$1$ law. 
We start our investigation of embedding-closed quantifiers by exhibiting this proof here.  

The notation $\mu(P) = r$ means that the asymptotic probability of a property $P$ is $r$.
A structure $\mathfrak{A}$ is \emph{homogeneous} if every isomorphism between finitely generated substructures of $\mathfrak{A}$ can be extended to an atomorphism of $\mathfrak{A}$.
Let $\tau$ be a relational vocabulary.
The \emph{random $\tau$-structure} is the unique homogeneous countable $\tau$-structure into which any finite $\tau$-structure can be embedded.
The following is a well-known fact:

\begin{theorem}
If $P$ is $\mathcal{L}_{\omega\omega}$-definable property that is true in the random structure then $\mu(P) = 1$.
\end{theorem}

\begin{lemma}\label{random}
Let $\tau$ be a finite relational vocabulary,
$\mathfrak{A}$ a finite $\tau$-structure and $\overline{a} \in A^n$ for some $n<\omega$.
Suppose that $t$ is the atomic type of $\overline{a}$ and 
set
\[
	P = \{\mathfrak{B} \in \Str[\tau] \colon \text{for all } \overline{b} \in B^n, \text{ if } \mathfrak{B} \vDash t(\overline{b}) \text{ then } (\mathfrak{A},\overline{a}) \leq (\mathfrak{B},\overline{b}) \}.
\]
Then $\mu(P) = 1$.
\end{lemma}
\begin{proof}
Denote the random $\tau$-structure by $\mathfrak{R}$ . 
The structure $\mathfrak{A}$ is embeddable into $\mathfrak{R}$ by, say, an embedding $f$. Let $\overline{b}$ be a tuple in $R$ whose atomic type is $t$.
Since $\mathfrak{R}$ is homogeneous, there is an automorphism $h$ that takes $\overline{a}$ to $\overline{b}$.
Thus, $h \circ f$ is an embedding of $(\mathfrak{A},\overline{a})$ into $(\mathfrak{R},\overline{b})$, so $P$ holds in the random structure, and since $P$ is $\mathcal{L}_{\omega\omega}$-definable, we have $\mu(P) = 1$.
\end{proof}

\begin{lemma}[\cite{Dawar: 2010}]\label{asympt2}
Let $\tau$ be a finite relational vocabulary, $Q$ embedding-closed quantifier of finite width $k$ and $\psi_0,\dots,\psi_{k-1}$ quantifier-free $\tau$-formulas.
There is a quantifier-free $\tau$-formula $\vartheta$ such that
$\forall \overline{x}(\vartheta \leftrightarrow Q(\overline{y}_i\psi_i)_{i<k})$
has asymptotic probability $1$.
\end{lemma}

\begin{proof}
Write $\varphi := Q(\overline{y}_i\psi_i)_{i < k}$, and set
\begin{align*}
	\vartheta := \bigvee\{t \colon &t \text{ is an atomic type and } (\mathfrak{A},\overline{a}) \vDash t \land \varphi \\
	&\text{ for some finite } \tau \text{-structure } \mathfrak{A} \text{ and tuple }\overline{a}\}.
\end{align*}

We clearly have $\mathfrak{A} \vDash \forall \overline{x}(\varphi \rightarrow \vartheta)$ for all finite $\tau$-structures $\mathfrak{A}$. 
For the other direction, if $\vartheta$ is an empty disjunction then $\varphi$ defines the empty relation on all finite structures thus being equivalent to a quantifier-free formula. Therefore, assume that $\vartheta$ is not empty. 

For each type $t$ in $\vartheta$, choose a pair $(\mathfrak{A},\overline{a})_t$ such that $(\mathfrak{A},\overline{a})_t \vDash t \land \varphi$ and $\mathfrak{A}$ is finite. Such pairs exist by the definition of $\vartheta$.
Let $\mathfrak{B}$ be a finite $\tau$-structure such that $\mathfrak{B} \vDash \exists \overline{x}(\vartheta \land \neg \varphi)$.
Then there is a tuple $\overline{b}$ in $B$ such that $(\mathfrak{B},\overline{b}) \vDash \vartheta \land \neg \varphi$.
Let $t$ be the atomic type of $\overline{b}$.
Then $(\mathfrak{A},\overline{a})_t \vDash t \land \varphi$, so if $(\mathfrak{A},\overline{a}) \leq (\mathfrak{B},\overline{b})$ then by Lemma \ref{preservation}, $(\mathfrak{B},\overline{b}) \vDash \varphi$ which is contradiction.
Thus, $(\mathfrak{A},\overline{a}) \nleq (\mathfrak{B},\overline{b})$, so by Lemma \ref{random}, $\mu(\exists \overline{x}(\vartheta \land \neg \varphi)) = 0$, so $\mu(\forall \overline{x}(\vartheta \rightarrow \varphi)) = 1$.
\end{proof}

\begin{theorem}[\cite{Dawar: 2010}]
Let $\tau$ be a finite relational vocabulary and
$\mathcal{Q}$ a finite set of embedding-closed quantifiers of finite width.
For any $\tau$-formula $\varphi \in \mathcal{L}^{\omega}_{\infty\omega}(\mathcal{Q})$ there is a quantifier-free $\tau$-formula $\vartheta$ such that $\forall \overline{x}(\vartheta \leftrightarrow \varphi)$
has asymptotic probability $1$.
\end{theorem}
\begin{proof}
Let $k$ be a natural number. There are, up to logical equivalence, finitely many quantifier-free formulas that use only variables $x_0,\dots,x_{k-1}$. Let $\psi_0,\dots,\psi_{l-1}$ be an enumeration of all $\mathcal{L}^k_{\infty\omega}$-formulas of the form $Q(\overline{y}_i \vartheta_i)_{i < n}$ with all $\vartheta_i$ quantifier-free. Note that $l$ is finite. By Lemma \ref{asympt2}, for each $i < l$ there is a quantifier free formula $\chi_i$ such that $\forall \overline{x}(\psi_i \leftrightarrow \chi_i)$ has asymptotic probability $1$. For every $i < l$, let $C_i$ be the set of all isomorphism types of finite structures on which $\forall \overline{x}(\psi_i \leftrightarrow \chi_i)$ is true. Put $C = \bigcap_{i < l}C_i$. Then $\mu(C) = 1$, since $l$ is finite and $\mu(C_i) = 1$ for all $i$.

Now we can show that for all $\varphi \in \mathcal{L}^k_{\infty\omega}(\mathcal{Q})$ there is a quantifier-free formula $\vartheta$ such that $\mathfrak{A} \vDash \forall \overline{x}(\vartheta \leftrightarrow \varphi)$ for all $\mathfrak{A} \in C$ from which the claim follows.
We use induction on the structure of $\varphi$. If $\varphi$ is quantifier-free, there is nothing to prove.
It is also clear that the claim holds for $\varphi = \neg \alpha$ and for $\varphi = \bigwedge_{i \in I}\alpha_i$ if it holds for $\alpha$ and all $\alpha_i$, respectively.
Assume that $\varphi = Q(\overline{y}_i \alpha_i)_{i < n}$ and the claim holds for all $\alpha_i$.
By the induction hypothesis, there are quantifier-free formulas $\vartheta_i$ such that
\[
	\mathfrak{A} \vDash \forall \overline{x}(\varphi \leftrightarrow Q(\overline{y}_i\vartheta_i)_{i < n})
\]
for all $\mathfrak{A} \in C$.
Since $Q(\overline{y}_i\vartheta_i)_{i<n} = \psi_m$ for some $m < l$, we have 
$\forall \overline{x}(\varphi \leftrightarrow \chi_m)$ on all structures of $C_m$, and therefore on $C$, because $C \subseteq C_m$.
\end{proof}

\begin{corollary}[\cite{Dawar: 2010}]
For any finite set $\mathcal{Q}$ of embedding-closed quantifiers of finite width,
the logic $\mathcal{L}^{\omega}_{\infty\omega}(\mathcal{Q})$ has a $0$-$1$ law.
\end{corollary}

\section{Quantifier elimination for $\mathcal{L}_{\infty \omega}(\mathcal{Q}_{\emb})$ and \\ some undefinability results}\label{chain_section}

In this section we introduce a method that allows us to produce some undefinability results for logics with embedding-closed quantifiers that cannot be established by using a $0$-$1$ law. For instance we will show that equicardinality cannot be defined in $\mathcal{L}_{\infty \omega}(\mathcal{Q}_{\emb})$.

\begin{definition}
A structure $\mathfrak{A}$ is \emph{homogeneous} if every isomorphism between finitely generated substructures of $\mathfrak{A}$ can be extended to an automorphism of $\mathfrak{A}$. We say that $\mathfrak{A}$ is \emph{quasi-homogeneous} if every isomorphism between finitely generated substructures of $\mathfrak{A}$ can be extended to an embedding of $\mathfrak{A}$ into itself.
\end{definition}

Note that a structure $\mathfrak{A}$ is homogeneous (quasi-homogeneous) if and only if for all tuples $\overline{a}$ and $\overline{b}$ of the same atomic type there is an automorphism (embedding) of $\mathfrak{A}$ taking $\overline{a}$ to $\overline{b}$.
It is clear that
every countable quasi-homogeneous structure is homogeneous.
Let $\mathfrak{R} = (R\setminus\{r\},\leq)$ be the usual ordering of real numbers with some number $r$ removed.
This "punctured" real line is an example of a structure that is quasi-homogeneous but not homogeneous.

\begin{definition}
Suppose $\mathcal{L}$ is a logic.
We say that a structure $\mathfrak{A}$ has \emph{quantifier elimination for} $\mathcal{L}$ if for all formulas $\varphi \in \mathcal{L}$ there is a quantifier-free formula $\vartheta$ such that $\mathfrak{A} \vDash \forall \overline{x}(\vartheta \leftrightarrow \varphi)$.
\end{definition}

\begin{theorem}\label{homog}
A $\tau$-structure $\mathfrak{A}$ has quantifier elimination for $\mathcal{L}_{\infty\omega}(\mathcal{Q}_{\emb})$ if and only if it is quasi-homogeneous.
\end{theorem}
\begin{proof}
Assume for simplicity that $\tau$ is relational vocabulary.
The proof can be generalized in a straightforward way to vocabularies with constant and function symbols.
Suppose first that $\mathfrak{A}$ has quantifier elimination.
Let $\overline{a} = (a_1,\dots,a_k)$ and $\overline{b} = (b_1,\dots,b_k)$ be tuples of elements of $A$ having the same atomic type.
We want to find embedding of $\mathfrak{A}$ into itself that maps $\overline{a}$ to $\overline{b}$.
Let $\tau' = \tau \cup \{P\}$ where $P$ is a new relation symbol of arity $k$.
Define a $\tau'$-structure $\mathfrak{A}'$ by setting $\mathfrak{A}'\upharpoonright \tau = \mathfrak{A}$ and $P^{\mathfrak{A}'} = \{\overline{a}\}$.
Let $Q$ be a quantifier whose defining class is
\[
	K_Q = \{\mathfrak{B} \in \Str[\tau'] \colon \mathfrak{A}' \text{ is embeddable into } \mathfrak{B}\},
\]
and suppose $\varphi \in \mathcal{L}_{\infty\omega}(\mathcal{Q}_{\emb})[\tau]$ is the next formula:
\[
	\varphi(\overline{z}) := Q((\overline{x}_RR)_{R\in\tau},\overline{x}_P=\overline{z}).
\]
Then $\mathfrak{A} \vDash \varphi(\overline{a})$ and, since $\mathfrak{A}$ has quantifier elimination and $\overline{a}$ and $\overline{b}$ have the same atomic type, we have $\mathfrak{A} \vDash \varphi(\overline{b})$, so there is an embedding $f$ of $\mathfrak{A}'$ into $(A,(S^{\mathfrak{A}})_{S\in\tau}, \{\overline{b}\})$ which clearly is a wanted embedding.

For the other direction, assume that $\mathfrak{A}$ is quasi-homogeneous. Let $Q \in \mathcal{Q}_{\emb}$ and suppose $(\psi_R)_{R \in \tau_Q}$ are quantifier-free formulas.
Now set $\varphi := Q(\overline{x}_R\psi_R)_{R \in \tau_Q}$ and denote by $k$ the number of free variables of $\varphi$.
Let
\[
	\vartheta = \bigvee\{t \colon t \text{ is an atomic type and for some } \overline{a} \in A^k, (\mathfrak{A}, \overline{a}) \vDash \varphi \land t\}.
\]
Let $\overline{b} \in A^k$ and $(\mathfrak{A}, \overline{b}) \vDash \vartheta$.
Then $(\mathfrak{A},\overline{b}) \vDash t$ and $(\mathfrak{A},\overline{a}) \vDash \varphi \land t$ for some $\overline{a} \in A^k$ with atomic type $t$.
Since $\mathfrak{A}$ is quasi-homogeneous, there is an embedding $f$ of  $(\mathfrak{A},\overline{a})$ into $(\mathfrak{A},\overline{b})$.
Since $\mathfrak{A},\overline{a} \vDash \varphi$, we have $(A,(\psi^{\mathfrak{A},\overline{a}}_R)_{R \in \tau_Q}) \in K_Q$,
so $(A,(\psi^{\mathfrak{A},\overline{b}}_R)_{R \in \tau_Q}) \in K_Q$
because $Q$ is embedding-closed and $\psi_R$ are quantifier-free.
Thus, $\mathfrak{A} \vDash \forall \overline{x}(\vartheta \rightarrow \varphi)$,
and since clearly $\mathfrak{A} \vDash \forall \overline{x}(\varphi \rightarrow \vartheta)$, we have $\mathfrak{A} \vDash \forall \overline{x}(\vartheta \leftrightarrow \varphi)$.

Thus, by using induction, we can eliminate quantifiers in all formulas $\varphi \in \mathcal{L}_{\infty\omega}(\mathcal{Q}_{\emb})$.
\end{proof}

\subsection{The finite case}

In this subsection, we will consider logics $\mathcal{L}_{\infty\omega}^m$ (the restriction of $\mathcal{L}_{\infty\omega}$ to formulas that use at most $m$ variables) extended with finite number of em\-bed\-ding-closed quantifiers of finite width.
We will show that in a countably infinite chain of quasi-homogeneous structures of finite relational vocabulary,
a formula of such a logic is eventually equivalent to a quantifier-free formula.
This will allow us to demonstrate, among other things, that certain properties of finite structures are not definable in such a logic.

\begin{lemma}\label{quant_elim_finite_lemma}
Let $\tau$ be a finite relational vocabulary, 
$Q$ an embedding-closed quantifier of width $n<\omega$ and $\varphi = Q(x_i\psi_i)_{i<n}$ where all $\psi_i$ are quantifier-free $\tau$-formulas.
Let $(\mathfrak{A}_i)_{i<\omega}$ be a chain of quasi-homogeneous $\tau$-structures. 
Then there is a natural number $k$ and a quantifier-free $\tau$-formula $\vartheta$ such that
\[
	\mathfrak{A}_i \vDash \forall \overline{x}(\varphi \leftrightarrow \vartheta)
\]
for all $k \leq i$.
\end{lemma}
\begin{proof}
For each $i < \omega$, let 
\[
	T_i = \{t \colon t \text{ is an atomic type and } (\mathfrak{A}_i,\overline{a}) \vDash \varphi \land  t \text{ for some } \overline{a} \}.
\]
Since all $\mathfrak{A}_i$ are quasi-homogeneous, it follows from Theorem \ref{homog} and Lemma \ref{quant} that 
\[
	\mathfrak{A}_i \vDash \forall \overline{x}(\varphi \leftrightarrow \bigvee T_i).
\]
Now let $i \leq j < \omega$ and $t \in T_i$. We have $(\mathfrak{A}_i,\overline{a}) \vDash \varphi \land t$ for some $\overline{a}$, so $(\mathfrak{A}_j,\overline{a}) \vDash \varphi \land t$ since both $\varphi$ and $t$ are preserved by embeddings. Thus, $t \in T_j$, so $T_i \subseteq T_j$ always when $i \leq j$.
Since there are finitely many distinct atomic $n$-types, the chain $(T_i)_{i < \omega}$ reaches its maximum at some $k < \omega$. Then $\vartheta = \bigvee T_k$ is a quantifier-free $\tau$-formula we want.
\end{proof}

\begin{theorem}\label{finite}
Let $\tau$ be a finite relational vocabulary, $\mathcal{Q}$ a finite set of embed\-ding-closed quantifiers of finite width and 
$(\mathfrak{A}_i)_{i<\omega}$ a chain of quasi-ho\-mo\-ge\-neous $\tau$-structures.
For each $m< \omega$, there is a natural number $N_m$ such that for every formula $\varphi \in \mathcal{L}^m_{\infty \omega}(\mathcal{Q})[\tau]$ there is a quantifier-free $\tau$-formula $\vartheta_{\varphi}$ such that
\[
	\mathfrak{A}_i \vDash \forall \overline{x}(\varphi \leftrightarrow \vartheta_{\varphi})
\]
for all $i \geq N_m$.
\end{theorem}
\begin{proof}
Let $\psi_0,\dots,\psi_l$ be an enumeration of all (up to equivalence) the $\tau$-formulas in at most $m$ variables having form $Q(\overline{x}_i\vartheta_i)_{i<n}$ with $Q \in \mathcal{Q}$ and all $\vartheta_i$ quantifier-free.
Note that $l$ is finite.
By Lemma \ref{quant_elim_finite_lemma} for each $\psi_i$ there is $k_i < \omega$ and a quantifier-free $\tau$-formula $\vartheta$ such that
\[
	\mathfrak{A}_j \vDash \forall \overline{x} (\psi_i \leftrightarrow \vartheta)
\]
when $j \geq k_i$.
We claim that we can set $N_m := \max\{k_i : i \leq l\}$. 
We prove the claim by induction on the structure of the formula $\varphi$.
The cases of $\varphi$ atomic, $\varphi = \neg \alpha$ and $\varphi = \bigwedge_{i \in I}\alpha_i$ are clear. Suppose $\varphi = Q(\overline{x}_i\alpha_i)_{i<n}$ and the claim holds for all $\alpha_i$.
Then there are quantifier-free $\tau$-formulas $\vartheta_i$ such that 
\[
	\mathfrak{A}_j \vDash \forall \overline{y}(\varphi \leftrightarrow Q(\overline{x}_i\vartheta_i)_{i<n})
\]
when $j \geq N_m$. Thus, if $j \geq N_m$ then $\varphi$ is equivalent to some $\psi_r$ and therefore to some quantifier-free formula $\vartheta$, so we can set $\vartheta_{\varphi} := \vartheta$.
\end{proof}

\begin{definition}
Let $\mathcal{L}$ be a logic, $\tau$ a vocabulary and $\mathfrak{A}$, $\mathfrak{B}$ $\tau$-structures.
Suppose also that there is an embedding $f$ of $\mathfrak{A}$ into $\mathfrak{B}$.
Let 
\[
	\tau^* = \tau \cup_{\disjoint} \{c_a : a \in A\}
\] 
and denote by $\mathfrak{A}^*$ and $\mathfrak{B}^*$ the $\tau^*$-extensions of $\mathfrak{A}$ and $\mathfrak{B}$, respectively, such that $c_a^{\mathfrak{A}^*} = a$ and $c_a^{\mathfrak{B}^*} = f(a)$ for all $a \in A$. 
Then we write $\mathfrak{A} \preceq_{\mathcal{L}} \mathfrak{B}$ if $\mathfrak{A}^*$ and $\mathfrak{B}^*$ satisfy exactly the same sentences of $\mathcal{L}[\tau^*]$.
\end{definition}

\begin{corollary}\label{finite2}
Let $\tau$ be a finite relational vocabulary, $\mathcal{Q}$ a finite set of embed\-ding-closed quantifiers of finite width and 
$(\mathfrak{A}_i)_{i<\omega}$ a chain of quasi-ho\-mo\-ge\-neous $\tau$-structures.
Let $m<\omega$ and 
write $\mathcal{L} = \mathcal{L}_{\infty\omega}^m(\mathcal{Q})$.
There is a natural number $N_m$ such that
$\mathfrak{A}_i \preceq_{\mathcal{L}} \mathfrak{A}_j$ for all $j \geq i \geq N_m$.
\end{corollary}

In Section \ref{law} we saw that logic $\mathcal{L}^{\omega}_{\infty\omega}$ extended with finitely many em\-bed\-ding-closed quantifiers of finite width has a $0$-$1$ law
which implies undefinability of certain properties, like having even cardinality, in such a logic.
By using Corollary \ref{finite2} we can determine further properties of finite structures that are not definable in 
a logic of this kind.
In order to apply the theorem, however, we first need to know which structures are homogeneous.
This question has been studied to some extent (a survey can be found in \cite{Lachlan: 1986}).
Finite homogeneous structures have been classified completely at least in the cases
of finite graphs \cite{Gardiner: 1976}, groups \cite{Cherlin: 2000} and rings \cite{Saracino: 1988}.
In addition, it is easy to see that all unary structures are homogeneous.

\begin{example}\label{hom_graphs}
According to \cite{Lachlan: 1986}, the only finite homogeneous (undirected) graphs are
up to complement
\begin{enumerate}
\item $Pe = (\{0,1,2,3,4\},\{ (i,j)\colon |i-j|\in\{1,4 \}\})$ (pentagon),
\item $K_3 \times K_3$,
\item $I_m[K_n]$ with $m,n < \omega$,
\end{enumerate}
where $K_n$ is the complete graph of $n$ vertices and $I_m[G]$ consists of $m$ disjoint copies of $G$.

It is easy to see that $I_m[K_n] \leq I_{m'}[K_{n'}]$ if and only if $m \leq m'$ and $n \leq n'$.
Thus, for example, the graph properties "there is a clique of even cardinality" or 
"there are more cliques than there are vertices in any clique" are not definable in $\mathcal{L}^{\omega}_{\infty\omega}(\mathcal{Q})$ for any finite set $\mathcal{Q}$ of embedding-closed quantifiers of finite width.
\end{example}

\begin{example}\label{Hartig_ex}
Let $\tau = \{U,V\}$ be a vocabulary with $U$ and $V$ unary relation symbols.
The quantifier corresponding to the class 
\[
	I = \{\mathfrak{A}\in\Str[\tau]\colon |U^{\mathfrak{A}}|=|V^{\mathfrak{A}}| \}
\]
of $\tau$-structures is known as \emph{H\"artig quantifier}.
Let $\mathcal{Q}$ be a finite set of embedding-closed quantifiers of finite width.
The following is a simple observation showing that the H\"artig quantifier is not definable in $\mathcal{L}^{\omega}_{\infty\omega}(\mathcal{Q})$ even if we consider only finite structures.

For all $i<\omega$ define a $\tau$-structure $\mathfrak{A}_i$ by setting 
$A_i = \{0,\dots,i\}$ and letting $U^{\mathfrak{A}_i}$ to be the set of all even and $V^{\mathfrak{A}_i}$ of all odd numbers of $A_i$.
Then $\mathfrak{A}_i \leq\mathfrak{A}_{i+1}$ for all $i<\omega$, and 
$|U^{\mathfrak{A}_i}|=|V^{\mathfrak{A}_i}|$ if and only if $i$ is odd.
Thus, it follows from Corollary \ref{finite2} that the class
\[
	\{\mathfrak{A} \in \Str[\tau] \colon \mathfrak{A} \text{ is finite and }|U^{\mathfrak{A}}| = |V^{\mathfrak{A}}| \}
\]
is not definable in the logic $\mathcal{L}^{\omega}_{\infty \omega}(\mathcal{Q})$.
In the next subsection we will show that in fact the H\"artig quantifier is not definable in the logic $\mathcal{L}_{\infty\omega}(\mathcal{Q}_{\emb})$ as well.
\end{example}

As we saw in Example \ref{hom_graphs}, there are only few homogeneous finite graphs,
so we cannot usually directly apply Corollary \ref{finite2} in studying definability
of graph properties.
The following theorem shows that this situation can be remedied to some extent by using \emph{interpretations}.

\begin{theorem}\label{inter}
Let $\tau$ and $\sigma$ be finite relational vocabularies and suppose $\mathcal{Q}$ is a finite set of embedding-closed quantifiers of finite width.
Let $m < \omega$, and write $\mathcal{L} = \mathcal{L}^m_{\infty\omega}(\mathcal{Q})$. 
Suppose $(\mathfrak{A}_i)_{i<\omega}$ is a chain of quasi-homogeneous $\tau$-structures and 
$(\Psi,(\psi_R)_{R\in\sigma})$ is an $\mathcal{L}$-interpretation of $\sigma$ in $\tau$.
There is a natural number $N_m$ such that for all $j \geq i \geq N_m$ we have $\Psi(\mathfrak{A}_i) \preceq_{\mathcal{L}} \Psi(\mathfrak{A}_j)$.
\end{theorem}
\begin{proof}
By Corollary \ref{finite2}, there is a natural number $N'_m$ such that $\mathfrak{A}_i \preceq_{\mathcal{L}} \mathfrak{A}_j$ for all $j\geq i \geq N'_m$.
We claim that $N'_m$ is a wanted number $N_m$.
To show that, let $\varphi \in \mathcal{L}[\sigma]$.
By Lemma \ref{interpr}, there is a formula $\varphi^* \in \mathcal{L}[\tau]$ such that
\[
	\mathfrak{A}_i, \overline{a} \vDash \varphi^* \Leftrightarrow \Psi(\mathfrak{A}_i), \overline{a} \vDash \varphi
\]
for all $i<\omega$ and tuples $\overline{a}$ of elements of $A$.
Let $j \geq i \geq N'_m$ and suppose that $\Psi(\mathfrak{A}_i),\overline{a} \vDash \varphi$.
Then we have $\mathfrak{A}_i,\overline{a} \vDash \varphi^*$, so $\mathfrak{A}_j, \overline{a} \vDash \varphi^*$ from which follows $\Psi(\mathfrak{A}_j), \overline{a} \vDash \varphi$. The claim is thus proved. 
\end{proof}

\begin{example}
A graph $G$ is \emph{regular} if every vertex of $G$ has the same number of neighbours.
We will show that regularity of a graph is not definable in the logic $\mathcal{L} = \mathcal{L}^{\omega}_{\infty \omega}(Q_1,\dots,Q_n)$ where all $Q_i \in \mathcal{Q}_{\emb}$ have finite width.
Let $\tau$ be the same vocabulary and $(\mathfrak{A}_i)_{i<\omega}$ a chain of $\tau$-structures as in Example \ref{Hartig_ex}.
Let $\sigma = \{E\}$ where $E$ is a relation symbol and suppose
$(\Psi,\psi)$ is an interpretaion of $\sigma$ in $\tau$ where
\[
	\psi(x,y) := (U(x)\land U(y)) \lor (V(x) \land V(y)).
\]
Then for all $k<\omega$,
\[
	\Psi(\mathfrak{A}_k) = \begin{cases}
	K_{\frac{k+1}{2}} + K_{\frac{k+1}{2}} \text{ if } k \text{ is odd} \\
	K_{\frac{k}{2}} + K_{\frac{k}{2}+1} \text{ if } k \text{ is even},
	\end{cases}
\]
where every $K_m$ is the complete graph on $m$ vertices and $+$ means disjoint union.
Thus, a graph $\Psi(\mathfrak{A}_k)$ is regular if and only if $k$ is odd, 
so by Theorem \ref{inter} regularity of graphs is not definable in $\mathcal{L}$.
\end{example}

\begin{example}\label{groups}
If we allow $\tau$ to have function symbols then Corollary \ref{finite2} does not hold.
Let $\tau$ be the vocabulary of groups.
A formula $\varphi \in \mathcal{L}_{\infty\omega}^{\omega}[\tau]$,
\begin{align*}
	\varphi(x,y) := \bigvee_{k<\omega}\big(&(x^k=1 \wedge y^k=1) \\
	&\wedge \neg \bigvee_{m<k}(x^m = 1) \wedge \neg \bigvee_{m<k}(y^m=1)\big)
\end{align*}
says that $x$ and $y$ have the same order, and a formula $\chi \in \mathcal{L}_{\infty\omega}^{\omega}[\tau]$,
\[
	\chi(x,y) := \bigvee_{k<\omega}y = x^k,
\]
says that $y$ is in the subgroup generated by $x$.
For each $n<\omega$, set
\begin{align*}
	G_{2n} &= \prod_{i\leq n} C_{p_i}^2 \\
	G_{2n+1} &= \prod_{i\leq n} C_{p_i}^2 \times C_{p_{n+1}},
\end{align*}
where $p_i$ is the $i$:th prime number and $C_{p_i}$ the cyclic group of order $p_i$.
Then $(G_n)_{n<\omega}$ is a chain of homogeneous groups, but $G_{2n} \vDash \psi$ and $G_{2n+1} \nvDash \psi$ for all $n$, where
\[
	\psi := \forall x \exists y(\neg\chi(x,y) \wedge \varphi(x,y)).
\]
\end{example}

We will need the following lemmas in the proof of Theorem \ref{log_eq}.
The proof of Lemma \ref{shrinking} follows from the main results of papers \cite{Cherlin: 1986} and \cite{Lachlan: 1984}. We will not present it here due to its complexity. 

\begin{lemma}\label{shrinking}
Let $\tau$ be a finite relational vocabulary and $\mathcal{H}$ the class of all finite homogeneous $\tau$-strucutres. Then all antichains of structures in $\mathcal{H}$ are finite, in other words, if $C \subseteq \mathcal{H}$ is infinite then there are $\mathfrak{A}$, $\mathfrak{B} \in C$ such that $\mathfrak{A} < \mathfrak{B}$.
\end{lemma}

\begin{lemma}\label{substr}
Let $\tau$ be a finite relational vocabulary and $\mathfrak{A}_0, \dots, \mathfrak{A}_{k-1}$ finite $\tau$-structures.
There is a sentence $\varphi \in \mathcal{L}_{\omega\omega}[\tau]$ such that for all $\mathfrak{B} \in \Str[\tau]$,
\[
	\mathfrak{B} \vDash \varphi \Leftrightarrow \mathfrak{A}_i \leq \mathfrak{B} \text{ for some } i < k.
\]
\end{lemma}
\begin{proof}
Every finite $\tau$-structure can be described up to isomorphism by a sentence in $\mathcal{L}_{\omega\omega}[\tau]$. Thus, if $\psi_i \in \mathcal{L}_{\omega\omega}[\tau]$ describes $\mathfrak{A}_i$ then $\bigvee_{i<k}\psi_i$ is the desired sentence.
\end{proof}

\begin{theorem}\label{log_eq}
Let $\tau$ be a finite relational vocabulary and $\mathcal{H}$ the class of all finite homogeneous $\tau$-structures. Suppose that $\mathcal{Q}$ is a finite set of embedding-closed quantifiers of finite width. Then $\mathcal{L}_{\infty\omega}^{\omega}(\mathcal{Q}) \equiv \mathcal{L}_{\omega\omega}$ over $\mathcal{H}$.
\end{theorem}
\begin{proof}
Let $\varphi \in \mathcal{L}^{\omega}_{\infty\omega}(\mathcal{Q})[\tau]$.
We say that $\mathfrak{A} \in \mathcal{H}$ \emph{stabilizes $\varphi$ in $\mathcal{H}$} if for all $\mathfrak{B} \in \mathcal{H}$, $\mathfrak{A} \leq \mathfrak{B}$ implies $\mathfrak{A} \vDash \varphi \Leftrightarrow \mathfrak{B} \vDash \varphi$.
Clearly, there is a structure $\mathfrak{A}_0 \in \mathcal{H}$ that stabilizes $\varphi$ in $\mathcal{H}$
since otherwise we would be able to construct an infinite chain of structures of $\mathcal{H}$ which contradicts Theorem \ref{finite}.
Similarly, if the class of finite $\tau$-structures incomparable with $\mathfrak{A}_0$ is not empty then we can find a structure $\mathfrak{A}_1 \in \mathcal{H}$ incomparable with $\mathfrak{A}_0$ that stabilizes $\varphi$ in $\mathcal{H}$.
Continuing in the same way we can construct an antichain $\mathcal{C} \subseteq \mathcal{H}$ of finite structures such that every structure $\mathfrak{B} \in \mathcal{H}$ is comparable with some structure $\mathfrak{A} \in \mathcal{C}$ 
and every structure in $C$ stabilizes $\varphi$.
By Lemma \ref{shrinking}, $\mathcal{C}$ is finite so by Lemma \ref{substr}, $\varphi$ is equivalent to a sentence in $\mathcal{L}_{\omega\omega}[\tau]$ over the structures in $\mathcal{H}$.
\end{proof}

In particular, if $\tau$ is a finite unary vocabulary and $\mathcal{Q}$ is a finite set of embedding-closed quantifiers of finite width then $\mathcal{L}_{\infty\omega}^{\omega}(\mathcal{Q}) \equiv \mathcal{L}_{\omega\omega}$ over finite $\tau$-structures.

\subsection{The infinite case}

It is possible to generalize Theorem \ref{finite} to vocabularies and sets of embedding-closed quantifiers of arbitrary cardinality.

\begin{theorem}\label{infinite}
Let $\tau$ be a vocabulary, $\kappa$ a cardinal, $\mathcal{Q}$ a set of embedding-closed quantifiers of width less than $\kappa$, and $\lambda$ a regular cardinal such that 
\[
	2^{|\tau|\cdot\aleph_0} \cdot |\mathcal{Q}| \cdot \kappa < \lambda.
\]
Suppose $(\mathfrak{A}_{\alpha})_{\alpha<\lambda}$ is a chain of quasi-homogeneous $\tau$-structures.
There is a cardinal $\mu<\lambda$ such that for every formula $\varphi \in \mathcal{L}_{\infty\omega}(\mathcal{Q})$ there is a quantifier-free $\tau$-formula $\vartheta_{\varphi}$ such that 
\[
	\mathfrak{A}_{\alpha} \vDash \forall \overline{x}(\varphi \leftrightarrow \vartheta_{\varphi})
\]
for all $\mu \leq \alpha$. 
In particular, $\mathfrak{A}_{\alpha} \preceq_{\mathcal{L}_{\infty\omega}(\mathcal{Q})} \mathfrak{A}_{\beta}$ for all $\mu \leq \alpha \leq \beta < \lambda$. 
\end{theorem}
\begin{proof}
Let $Q \in \mathcal{Q}_{\emb}$ and 
$\varphi = Q(\overline{x}_i\vartheta_i)_{i<\delta}$ where all $\vartheta_i$ are quantifier-free $\tau$-formulas.
In the first part of this proof, we will generalize the proof of Lemma \ref{quant_elim_finite_lemma} to apply for chains of quasi-homogeneous structures of arbitrary length.
First we observe that for all $\alpha< \lambda$ there is the smallest set $T_{\alpha}$ of atomic types of $\tau$ such that $\mathfrak{A}_{\alpha} \vDash \forall\overline{x}(\varphi \leftrightarrow \bigvee T_{\alpha})$ because of quasi-homogeneity of $\mathfrak{A}_{\alpha}$. 
In addition, if $\alpha < \beta < \lambda$ then $T_{\alpha} \subseteq T_{\beta}$.
Thus, since there are at most $2^{|\tau|+\aleph_0}$ atomic types of $\tau$, if $\lambda$ is a regular cardinal greater than $2^{|\tau|+\aleph_0}$
and $(\mathfrak{A}_{\alpha})_{\alpha<\lambda}$ is a chain of quasi-homogeneous $\tau$-structures
then there is a cardinal $\kappa_{\varphi}<\lambda$ and a quantifier-free $\tau$-formula $\vartheta_{\varphi}$ that is equivalent to $\varphi$ in structures $\mathfrak{A}_{\alpha}$ with $\alpha\geq\kappa_{\varphi}$.

Let $\Phi$ be the set of all formulas of the form $Q(\overline{x}_i\vartheta_i)_{i<\delta}$,
with $Q\in\mathcal{Q}$ and all $\vartheta_i$ quantifier-free.
Then $|\Phi| \leq 2^{|\tau|\cdot\aleph_0}\cdot |\mathcal{Q}| \cdot \kappa$.
We claim that we can set $\mu := \sup\{\kappa_{\varphi} \colon \varphi \in \Phi \}$.
We use induction on the structure of the formula to prove this claim.
If $\varphi$ is atomic there is nothing to prove. If $\varphi = \bigwedge_{i\in I} \varphi_i$ and the claim holds for all $\varphi_i$ then $\vartheta_{\varphi} = \bigwedge_{i\in I}\vartheta_{\varphi_i}$. If $\varphi = \neg \varphi'$ and the claim is true for $\varphi'$ then $\vartheta_{\varphi} = \neg \vartheta_{\varphi'}$.
Finally, suppose that $\varphi = Q(\overline{x}_i\varphi_i)_{i<\delta}$ and the claim holds for all $\varphi_i$. Then each $\varphi_i$ is equivalent to a quantifier-free formula $\vartheta_i$ on structures $\mathfrak{A}_{\alpha}$ with $\alpha \geq \mu$ so by the first part of the proof, $\varphi$ is also equivalent to some quantifier-free formula $\vartheta_{\varphi}$ on these structures.
\end{proof}

\begin{example}\label{Hartig}
It follows directly from Theorem \ref{infinite} that
the H\"artig quantifier, introduced in Example \ref{Hartig_ex}, is not definable in $\mathcal{L}_{\infty\omega}(\mathcal{Q}_{\emb})$.
\end{example}

\clearpage

\begin{lemma}
If $\mathfrak{A}$ and $\mathfrak{B}$ are bi-embeddable quasi-homogeneous $\tau$-structures then $\mathfrak{A} \equiv_{\emb} \mathfrak{B}$.
\end{lemma}
\begin{proof}
We can build a chain of arbitrary length in which structures $\mathfrak{A}$ and $\mathfrak{B}$ alternate.
By Theorem \ref{infinite} the truth value of any sentence $\varphi \in \mathcal{L}_{\infty\omega}(\mathcal{Q}_{\emb})$ is eventually preserved in this chain, so $\mathfrak{A} \vDash \varphi \Leftrightarrow \mathfrak{B} \vDash \varphi$.
\end{proof}

\begin{example}\label{completeness}
Let $\eta = (0,1)$, that is $\eta$ is the open real line interval between $0$ and $1$, and $\xi = \eta \setminus \{\frac{1}{2}\}$. Then $\eta$ and $\xi$ are both quasi-homogeneous and bi-embeddable, so $\eta \equiv_{\emb} \xi$. Thus, the completeness of an ordering is not definable in $\mathcal{L}_{\infty \omega}(\mathcal{Q}_{\emb})$.
\end{example}

\begin{example}\label{cofinality}
We denote by $\omega_{\alpha}^{\omega_{\alpha}}$ the set of all functions $\omega_{\alpha} \rightarrow \omega_{\alpha}$.
Let $\aleph_{\alpha}$ be a regular cardinal, $\eta$ the lexicographic ordering of the set $\omega_{\alpha}^{\omega_{\alpha}}$, and $\xi$ the lexicographic ordering of the set $\omega \times \omega_{\alpha}^{\omega_{\alpha}}$. 
Then $\cf(\eta) = \aleph_{\alpha}$ and $\cf(\xi) = \aleph_0$, where $\cf$ means the cofinality of an ordering.
The orderings $\eta$ and $\xi$ are both quasi-homogeneous  and bi-embeddable, hence $\eta \equiv_{\emb} \xi$. 
Therefore, for any ordinal $\beta$,
the property of having cofinality $\aleph_{\beta}$ is
not definable in $\mathcal{L}_{\infty \omega}(\mathcal{Q}_{\emb})$.
\end{example}

We can use Example \ref{cofinality} to obtain the following result:

\begin{theorem}
The logic $\mathcal{L}_{\infty\omega}(\mathcal{Q}_{\emb})$ does not allow interpolation for the logic $\mathcal{L}_{\omega\omega}(Q_1)$. 
\end{theorem}
\begin{proof}
Let $C_0$, $C_1$ be the classes of all linear orderings of cofinality $\aleph_0$, $\aleph_1$, respectively. 
Both $C_0$ and $C_1$ are projective classes in $\mathcal{L}_{\omega\omega}(Q_1)$.
To see this, let $\psi(\leq,U)$ be the sentence saying that $\leq$ is a linear ordering of the universe without a greatest element and $U$ induces a cofinal subordering of $\leq$. Furthermore, let $\chi(\leq,U)$ be the following sentence:
\[
	\chi(\leq,U) := \ Q_1xU(x)
	\wedge \forall x\big(U(x)\rightarrow \neg Q_1y\big(U(y) \wedge y \leq x\big) \big).
\]
We say that a linear ordering $\leq$ is \emph{$\omega_1$-like} if the sentence
\[
	Q_1(x=x) \wedge \forall x \neg Q_1 y(y\leq x)
\]
is true in it. 
Thus, if $\leq$ is a linear ordering then $\chi(\leq,U)$ says that the subordering of $\leq$ induced by $U$ is $\omega_1$-like. Now let
\[
	\varphi_0(\leq,U) := \psi(\leq,U) \wedge \neg Q_1 x U(x)
\]
and
\[
	\varphi_1(\leq,U) := \psi(\leq,U) \wedge \chi(\leq,U).
\]
Then clearly the sentence $\varphi_0$ is a projective definition of the class $C_0$.
That $\varphi_1$ defines projectively the class $C_1$ follows from the fact that an ordering has cofinality $\aleph_1$ if and only if it has a cofinal $\omega_1$-like subordering.

Thus, $C_0$ and $C_1$ are disjoint projective classes in $\mathcal{L}_{\omega\omega}(Q_1)$ that by Example \ref{cofinality} cannot be separated by any elementary class in $\mathcal{L}_{\infty\omega}(\mathcal{Q}_{\emb})$ from which the claim follows.
\end{proof}

\section{Embedding game}\label{game_section}

In this section we will introduce a game characterizing relation $\equiv_{\emb}$.
The \emph{embedding game} is played on two structures $\mathfrak{A}$ and $\mathfrak{B}$ of the same vocabulary by two players, Spoiler and Duplicator. A \emph{position} in the game is a tuple $(\mathfrak{A},\overline{a},\mathfrak{B},\overline{b})$, where $\overline{a}$ and $\overline{b}$ are tuples of elements of $A$ and $B$, respectively. The game proceeds in rounds and starts from the position $(\mathfrak{A},\emptyset,\mathfrak{B},\emptyset)$. 
Suppose that $n$ rounds of the game have been played and the position is $(\mathfrak{A},\overline{a},\mathfrak{B},\overline{b})$.
First Duplicator chooses embeddings $f \colon A \rightarrow B$ and $g \colon B \rightarrow A$ such that $f\overline{a} = \overline{b}$ and $g\overline{b} = \overline{a}$.
If there are no such embeddings then Spoiler wins the game.
Otherwise Spoiler selects a natural number $k$ and a tuple $\overline{c} \in A^k$ or $\overline{d} \in B^k$.
This completes the round, and the game continues from the position $(\mathfrak{A},\overline{a}\overline{c},\mathfrak{B},\overline{b}f\overline{c})$ or $(\mathfrak{A},\overline{a}g\overline{d},\mathfrak{B},\overline{b}\overline{d})$ depending on whether Spoiler chose $\overline{c} \in A^n$ or $\overline{d} \in B^n$. 
Duplicator wins the game if and only if the game goes on infinitely.

\begin{definition}
For every formula $\varphi \in \mathcal{L}_{\infty\omega}$, we define inductively its \emph{quantifier rank} denoted by $\qr(\varphi)$ by setting
\begin{align*}
	\qr(\varphi) &= 0 \text{ if } \varphi \text{ is quantifier-free}, \\
	\qr(\neg\varphi) &= \qr(\varphi), \\
	\qr(\bigvee\Phi) &= \qr(\bigwedge\Phi) = \sup\{\qr(\varphi) : \varphi \in \Phi\}, \\
	\qr(Q(\overline{x}_{\delta}\varphi_{\delta})) &= \sup\{\qr(\varphi_{\delta}) : \delta < \kappa \} + 1.
\end{align*}
We write $\mathfrak{A} \simeq_{\emb} \mathfrak{B}$ if Duplicator wins the embedding game on $\mathfrak{A}$ and $\mathfrak{B}$, and $\mathfrak{A} \simeq_{\emb}^{\gamma} \mathfrak{B}$ if Duplicator does not lose in the first $\gamma$ rounds. 
We write 
$\mathfrak{A} \equiv_{\emb} \mathfrak{B}$ if $\mathfrak{A}$ and $\mathfrak{B}$ agree on all sentences of $\mathcal{L}_{\infty \omega}(\mathcal{Q}_{\emb})$.
Notation $\mathfrak{A} \equiv_{\emb}^{\gamma} \mathfrak{B}$ means that $\mathfrak{A}$ and $\mathfrak{B}$ agree on all the sentences of $\mathcal{L}_{\infty \omega}(\mathcal{Q}_{\emb})$ whose quantifier rank is $\leq \gamma$.
\end{definition}

\begin{remark}\label{bijective}
The embedding game was inspired by and bears some resemblance to Hella's \emph{bijective game} which was introduced in \cite{Hella: 1984}.
In it, Duplicator selects a bijection $f$ between two structures $\mathfrak{A}$, $\mathfrak{B}$ instead of a pair of embeddings.
Then Spoiler chooses a tuple $\overline{c} \in A^n$ where $n<\omega$ is a number whose value is fixed at the beginning of the game (we say that it is \emph{$n$-bijective game}). Duplicator loses if $c \mapsto f(c)$ is not a partial isomorphism on the elements of $\overline{c}$ between the structures $(\mathfrak{A},\overline{a})$ and $(\mathfrak{B},\overline{b})$ where $\overline{a}$ and $\overline{b}$ are elements chosen in the previous rounds of the game.
Otherwise the game continues to the next round from the position $(\mathfrak{A},\overline{a}\overline{c},\mathfrak{B},\overline{b}f\overline{c})$.
The $n$-bijective game characterizes the equivalence of structures in relation to the logic $\mathcal{L}_{\infty\omega}$ extended with the class of all generalized quantifiers of arity $\leq n$. 
\end{remark}

\begin{remark}
Let $\tau$ be a vocabulary, $\mathfrak{A}$ and $\mathfrak{B}$ $\tau$-structures, $\overline{a} \in A^n$ and $\overline{b} \in B^n$.
A position $(\mathfrak{A},\overline{a},\mathfrak{B},\overline{b})$ is equivalent to the position $(\mathfrak{A}',\emptyset,\mathfrak{B}',\emptyset)$, where $\mathfrak{A}'$ and $\mathfrak{B}'$ are structures of vocabulary $\tau$ expanded with new constant symbols $c_1,\dots,c_n$ with interpretations $c_i^{\mathfrak{A}'} = a_i$ and $c_i^{\mathfrak{B}'} = b_i$ for all $i$. 
For the sake of brevity, we will use vocabulary expansions instead of writing positions explicitely.  
\end{remark}

\begin{theorem}\label{game}
Let $\tau$ be vocabulary and $\mathfrak{A}$, $\mathfrak{B}$ $\tau$-structures.
For all ordinals $\gamma \geq 1$ we have $
	\mathfrak{A} \simeq_{\emb}^{\gamma} \mathfrak{B} \text{ if and only if } \mathfrak{A} \equiv_{\emb}^{\gamma} \mathfrak{B}$.
\end{theorem}
\begin{proof}
We use induction on $\gamma$. Supppose first that $\mathfrak{A} \simeq^1_{\emb} \mathfrak{B}$. Then $\mathfrak{A} \leq \mathfrak{B}$ and $\mathfrak{B} \leq \mathfrak{A}$, so $\mathfrak{A} \equiv_{\emb}^1 \mathfrak{B}$ by Lemma \ref{preservation}. Assume next that $\mathfrak{A} \equiv^1_{\emb} \mathfrak{B}$.
Let $\mathfrak{A}'$ be the structure $\mathfrak{A}$ with functions and constants replaced by corresponding relations. 
Let $Q$ be the smallest embedding-closed quantifier containing $\mathfrak{A}'$.
For each symbol of $\tau$ define $\varphi_R := R$ for all relation symbols $R \in \tau$, $\varphi_f := f(\overline{x}) = y$ for all function symbols $f \in \tau$ and $\varphi_c := x = c$ for all constant symbols $c \in \tau$. Then $\mathfrak{A} \vDash Q(\overline{x}_S \varphi_S)_{S \in \tau}$, and since $\mathfrak{A} \equiv^1_{\emb} \mathfrak{B}$, we have $\mathfrak{B} \vDash Q(\overline{x}_S \varphi_S)_{S \in \tau}$, so $\mathfrak{A} \leq \mathfrak{B}$. In the same way we prove that $\mathfrak{B} \leq \mathfrak{A}$, so $\mathfrak{A} \simeq^1_{\emb} \mathfrak{B}$. The base step of induction is thus proved. Assume now that $\gamma > 1$ and the claim holds for all $\alpha < \gamma$.

Suppose first that $\mathfrak{A} \simeq^{\gamma}_{\emb} \mathfrak{B}$.
We use induction on the structure of formulas to show that $\mathfrak{A} \equiv_{\emb}^{\gamma} \mathfrak{B}$. Thus, assume that $\mathfrak{A} \vDash \varphi$ and $\qr(\varphi) \leq \gamma$.
If $\varphi$ is quantifier-free then $\mathfrak{B} \vDash \varphi$ since otherwise Duplicator would lose immediately.
If $\varphi = \neg \psi$ or $\varphi = \bigwedge \Psi$ and the claim holds for $\psi$ and all $\chi \in \Psi$ then it is straightforward to see that $\mathfrak{B} \vDash \varphi$ as well.
Now assume that $\varphi = Q(\overline{x}_{\delta}\psi_{\delta})_{\delta<\kappa}$ and the claim is true for all $\psi_{\delta}$.
Let $f$ be an embedding $\mathfrak{A} \rightarrow \mathfrak{B}$ that 
Duplicator can choose in the first round according to her winning strategy.
Then for all $\overline{a}\in A^{<\omega}$ we have $(\mathfrak{A},\overline{a})\simeq^{\gamma-1}_{\emb}(\mathfrak{B},f\overline{a})$
if $\gamma$ is finite and $(\mathfrak{A},\overline{a})\simeq^{\gamma}_{\emb}(\mathfrak{B},f\overline{a})$
if $\gamma$ is infinite,
so $(\mathfrak{A},\overline{a})\simeq^{\alpha}_{\emb}(\mathfrak{B},f\overline{a})$
for all $\alpha<\gamma$.
Thus, by the induction hypothesis, $(\mathfrak{A},\overline{a})\equiv^{\alpha}_{\emb}(\mathfrak{B},f\overline{a})$
for all $\alpha<\gamma$ so 
\[
	\mathfrak{A} \vDash \psi_{\delta}(\overline{a}) \Leftrightarrow \mathfrak{B} \vDash \psi_{\delta}(f\overline{a})
\]
for all $\delta<\kappa$ since $\qr(\psi_{\delta}) < \gamma$ by the definition of quantifier rank and the fact that $\qr(\varphi) \leq \gamma$.
This means that $f$ is an embedding
\[
	(A, (\psi_{\delta}^{\mathfrak{A}})_{\delta < \kappa}) \rightarrow (B, (\psi_{\delta}^{\mathfrak{B}})_{\delta < \kappa}),
\]
so $\mathfrak{B} \vDash Q(\overline{x}_{\delta} \psi_{\delta})_{\delta < \kappa}$ since $Q$ is embedding-closed.
In the same way we show that $\mathfrak{B} \vDash \varphi$ implies $\mathfrak{A} \vDash \varphi$ for all $\varphi \in \mathcal{L}_{\infty\omega}(\mathcal{Q}_{\emb})$ with quantifier rank $\leq \gamma$ thus proving that $\mathfrak{A} \equiv^{\gamma}_{\emb} \mathfrak{B}$.

For the other direction, assume that $\mathfrak{A} \not \simeq^{\gamma}_{\emb} \mathfrak{B}$.
We denote by $\mathcal{F}_A$ the set of all embeddings $\mathfrak{A} \rightarrow \mathfrak{B}$, and by $\mathcal{F}_B$ the set of all embeddings $\mathfrak{B} \rightarrow \mathfrak{A}$.
Then for each pair of $(f,g) \in \mathcal{F}_A\times\mathcal{F}_B$
there are tuples $\overline{a} \in A^{<\omega}$, $\overline{b} \in B^{<\omega}$ such that $(\mathfrak{A},\overline{a}) \not \simeq^{\alpha}_{\emb} (\mathfrak{B},f\overline{a})$ or $(\mathfrak{A},g\overline{b}) \not \simeq^{\alpha}_{\emb} (\mathfrak{B},\overline{b})$ for some $\alpha < \gamma$.
Thus, by the induction hypothesis, for each pair of embeddings $(f,g)$ there is a formula $\psi_f$ or a formula $\psi_g$ of quantifier rank $< \gamma$, such that
\begin{align*}
	(*)\ &\mathfrak{A} \vDash \psi_f(\overline{a}) \nLeftrightarrow \mathfrak{B} \vDash \psi_f(f\overline{a}) \text{ or } \\
	&\mathfrak{A} \vDash \psi_g(g\overline{b}) \nLeftrightarrow \mathfrak{B} \vDash \psi_g(\overline{b})
\end{align*}
for some $\overline{a} \in A^{<\omega}$ or $\overline{b} \in B^{<\omega}$.
Let $Q_A$ be the smallest embedding-closed quantifier containing the structure
$(A, (\psi_h^{\mathfrak{A}})_{h \in \mathcal{F}_A})$, and $Q_B$ be the smallest embedding-closed quantifier containing the structure
$(B, (\psi_h^{\mathfrak{B}})_{h \in \mathcal{F}_B})$.
Then $\mathfrak{A} \vDash Q_A(\overline{x}_h \psi_h)_{h \in \mathcal{F}_A}$ and 
$\mathfrak{B} \vDash Q_B(\overline{x}_h \psi_h)_{h \in \mathcal{F}_B}$.
Assume to the contrary that  $\mathfrak{B} \vDash Q_A(\overline{x}_h \psi_h)_{h \in \mathcal{F}_A}$ and 
$\mathfrak{A} \vDash Q_B(\overline{x}_h \psi_h)_{h \in \mathcal{F}_B}$.
Then there are embeddings  
\begin{align*}
	f : (A,(\psi_h^{\mathfrak{A}})_{h \in \mathcal{F}_A}) &\rightarrow (B,(\psi_h^{\mathfrak{B}})_{h \in \mathcal{F}_A}) \text{ and} \\
	g : (B,(\psi_h^{\mathfrak{B}})_{h \in \mathcal{F}_B}) &\rightarrow (A,(\psi_h^{\mathfrak{A}})_{h \in \mathcal{F}_B}),
\end{align*}
so there are embeddings $f$ and $b$ such that
\begin{align*}
	\mathfrak{A} \vDash \psi_f(\overline{a}) &\Leftrightarrow \mathfrak{B} \vDash \psi_f(f\overline{a}) \text{ and }\\
	\mathfrak{A} \vDash \psi_g(g\overline{b}) &\Leftrightarrow \mathfrak{B} \vDash \psi_g(\overline{b})
\end{align*}
for all $\overline{a}$ and $\overline{b}$,
which contradicts $(*)$. 
Thus, $\mathfrak{B} \nvDash Q_A(\overline{x}_h \psi_h)_{h \in \mathcal{F}_A}$ or 
$\mathfrak{A} \nvDash Q_B(\overline{x}_h \psi_h)_{h \in \mathcal{F}_B}$,
which completes the induction step. 
\end{proof}

The next proposition says that in order to ensure that $\mathfrak{A} \equiv_{\emb} \mathfrak{B}$ it is sufficient for Duplicator to have a winning strategy for the embedding game of length $\omega$.

\begin{proposition}
For all structures $\mathfrak{A}$, $\mathfrak{B}$ of the same vocabulary we have
$\mathfrak{A} \equiv_{\emb} \mathfrak{B}$ if and only if $\mathfrak{A} \equiv_{\emb}^{\omega} \mathfrak{B}$.
\end{proposition}
\begin{proof}
The implication from left to right is trivial.
For the other direction, we use induction on ordinals $\gamma$ to show that for all structures $\mathfrak{A}$, $\mathfrak{B}$ if $\mathfrak{A} \equiv_{\emb}^{\omega} \mathfrak{B}$ then $\mathfrak{A} \equiv_{\emb}^{\gamma} \mathfrak{B}$.
Thus suppose that the claim holds for all $\alpha < \gamma$.
Let $\varphi$ be an $\mathcal{L}_{\infty\omega}(\mathcal{Q}_{\emb})$-sentence with $\qr(\varphi) \leq \gamma$.
We use induction on the structure of $\varphi$ to show that $\mathfrak{A}$ and $\mathfrak{B}$ agree on its truth value.
The only interesting case is when $\varphi = Q(\overline{x}_{\delta}\psi_{\delta})_{\delta<\kappa}$.
Suppose that $\mathfrak{A} \vDash \varphi$, and let $\alpha = \sup\{\qr(\psi_{\delta}) : \delta < \kappa \}$
By the definition of quantifier rank we have $\alpha < \gamma$ so by the induction hypthesis $\mathfrak{A} \simeq^{\alpha}_{\emb} \mathfrak{B}$.
If $\alpha$ is finite then $\mathfrak{B} \vDash \varphi$ since $\mathfrak{A} \equiv^{\omega}_{\emb} \mathfrak{B}$ and we are done.
Thus assume that $\alpha$ is infinite.  
Let $f : \mathfrak{A} \rightarrow \mathfrak{B}$ be an embedding that Duplicator can choose in the first round in order to win the game of length $\alpha$.
Then $(\mathfrak{A},\overline{a}) \simeq^{\alpha}_{\emb} (\mathfrak{B},f\overline{a})$ for all $\overline{a} \in A^{<\omega}$ since $\alpha$ is infinite so
\[
	\mathfrak{A} \vDash \psi_{\delta}(\overline{a}) \Leftrightarrow \mathfrak{B} \vDash \psi_{\delta}(f\overline{a}) 
\]
for all $\overline{a} \in A^{\omega}$. Hence, $f$ is an embedding 
\[
	(A,(\psi_{\delta}^{\mathfrak{A}})_{\delta<\kappa}) \rightarrow (B,(\psi_{\delta}^{\mathfrak{B}})_{\delta<\kappa})
\]
so $\mathfrak{B} \vDash \varphi$ since $Q$ is embedding-closed.
In the same way we show that $\mathfrak{B} \vDash \varphi$ implies $\mathfrak{A} \vDash \varphi$.
\end{proof}

\begin{example}\label{hier}
Let $E_0$ be an equivalence relation with countably infinite number of $E_0$-classes, and suppose that each $E_0$-class has cardinality $\aleph_1$.
Let $E_1$ satisfy the same conditions with exception of having one $E_1$-class of cardinality $\aleph_0$.
Then $E_0$ and $E_1$ are bi-embeddable, so $E_0 \equiv_{\emb}^1 E_1$.

Let $f \colon E_0 \rightarrow E_1$ and $g \colon E_1 \rightarrow E_0$ be embeddings.
Let $[a]_{E_1}$ be the $E_1$-class of cardinality $\aleph_0$,
and suppose Spoiler chooses the embedding $g$ and the element $a$. 
It is easy to see that there is no embedding of $E_0$ into $E_1$ that maps $g(a)$ to $a$, since the restriction of such an embedding to an $E_0$-class must be included in some $E_1$-class, and $|[g(a)]_{E_0}| = \aleph_1$ and $|[a]_{E_1}| = \aleph_0$.
Thus, Duplicator loses in the second round, so $E_0 \not \equiv_{\emb}^2 E_1$. 
\end{example}

We are going to use the embedding game in order to prove the following theorem:
\begin{theorem}\label{for_each}
For each $n<\omega$, there is a first-order sentence $\varphi_n$ of quantifier rank $n$ that is not expressible by any $\mathcal{L}_{\infty \omega}(\mathcal{Q}_{\emb})$-sentence of quantifier rank $<n$ and is of the form
\[
	\varphi_n = Q_n x_n \cdots Q_1 x_1 \vartheta(x_1,\dots,x_n)
\]
where 
\[
	Q_n = \begin{cases}
	\forall \text{ if } n \text{ is odd},\\
	\exists \text{ if } n \text{ if even},
\end{cases}
\]
%$Q_n = \exists$ if $n$ is odd, $Q_n = \forall$ if $n$ is even, 
and $\vartheta$ is quantifier-free.
\end{theorem}

For the remaining part of this text we will denote by $\eta$ the usual ordering of rational numbers. 
The following lemma is a well-known fact.

\begin{lemma}\label{order_isom}
Every open segment of $\eta$ is isomorphic to $\eta$.
\end{lemma}

\begin{definition}
In the text of this definition we will assume that all structures are disjoint unless mentioned otherwise.
For every natural number $n \geq 1$, we define the vocabulary $\tau_n = \{P_0,\dots,P_n,\leq,E_1,\dots,E_n\}$ where all $P_i$ are unary and $E_i$, $\leq$ binary relation symbols, and 
%$\tau_n^* = \tau_n \cup \{C_{\alpha} : \alpha \in \omega^{\omega}\}$
$\tau_n^* = \tau_n \cup \{C_1,\dots,C_n\}$ 
where all $C_i$ are unary relation symbols.
Our aim is to define classes $S_n, T_n \subset \Str[\tau_n^*]$ such that $(\mathfrak{A} \upharpoonright \tau_n) \equiv_{\emb}^n (\mathfrak{B} \upharpoonright \tau_n)$ but 
$(\mathfrak{A} \upharpoonright \tau_n) \not \equiv_{\mathcal{L}_{\omega\omega}}^{n+1} (\mathfrak{B} \upharpoonright \tau_n)$ for all $\mathfrak{A} \in S_n$ and $\mathfrak{B} \in T_n$.

$\textbf{n = 1:}$
First we define class $R_1 \subset \Str[\tau_1^*]$ by setting $\mathfrak{A} \in R_1$ if and only if there are disjoint structures $\mathfrak{B}_0$, $\mathfrak{B}_1 \in \Str[\{\leq\}]$ isomorphic to $\eta$ such that 
$\mathfrak{A} \upharpoonright \{\leq\} = \mathfrak{B}_0 \cup \mathfrak{B}_1$ and 
$P_i^{\mathfrak{A}} = B_i$ for $i = 0,1$.
In what follows, we will use $\eta^{\mathfrak{A}}_i$ to denote $\mathfrak{B}_i$ for $i = 0,1$. Note that $\eta_i^{\mathfrak{A}} = (\mathfrak{A} \upharpoonright \{\leq\}) | P_i^{\mathfrak{A}}$.
Finally we set
\[
	S_1 = \{\mathfrak{A}\in R_1 : E_1^{\mathfrak{A}} \text{ is an isomorphism between } \eta^{\mathfrak{A}}_1 \text{ and } \eta^{\mathfrak{A}}_0 \}
\]
and
\begin{align*}
	T_1 = \{\mathfrak{A}\in R_1 : &\  C_1^{\mathfrak{A}} = \{a\} \text{ for some } a \in P^{\mathfrak{A}}_1 \text{ and } \\
	& \text{ there is an isomorphism } h \text{ between }\\ &\ \eta^{\mathfrak{A}}_1 \text{ and }\eta^{\mathfrak{A}}_0 \text{ such that } E_1^{\mathfrak{A}} = h \setminus \{(a,h(a))\}.
\end{align*}

\includegraphics[scale = 0.8]{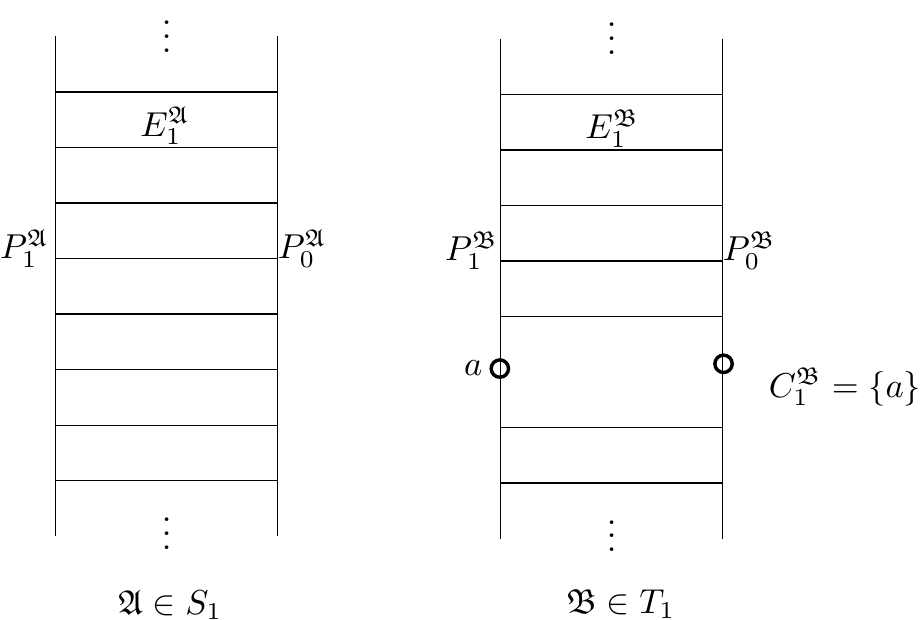}

$\textbf{n > 1:}$ We define class $R_n \subset \Str[\tau_n^*]$ by setting $\mathfrak{A} \in R_n$ if and only if there are structures $\mathfrak{B} \in \Str[\{\leq\}]$ and $\mathfrak{M}_a \in \Str[\tau^*_{n-1}]$ for each $a \in B$ such that
\[
	\mathfrak{A} \upharpoonright \tau^*_{n-1}  = \mathfrak{B} \cup \bigcup_{a \in B}\mathfrak{M}_a,
\] 
$P_n^{\mathfrak{A}} = B$, $C_n^{\mathfrak{A}} = \{a\}$ for some $a \in P_n^{\mathfrak{A}}$, and
\[
	E_n^{\mathfrak{A}} = \{(a,b) \in A^2 : a \in P_n^{\mathfrak{A}} \text{ and } b \in M_a \}.
\]
If $n$ is even then we set $\mathfrak{A}\in S_n$ if and only if $\mathfrak{A} \in R_n$ and for all $a\in P^{\mathfrak{A}}_n$,
\begin{align*}
	\mathfrak{M}_a \in T_{n-1} &\text{ if } a \notin C^{\mathfrak{A}}_n,\\
	\mathfrak{M}_a \in S_{n-1} &\text{ if } a \in C^{\mathfrak{A}}_n,
\end{align*}
and $\mathfrak{A} \in T_n$ if and only if $\mathfrak{A} \in R_n$ and for all $a \in P_n^{\mathfrak{A}}$ we have
$\mathfrak{M}_a \in T_{n-1}$.
If $n$ is odd then we set $\mathfrak{A} \in S_n$ if and only if $\mathfrak{A} \in R_n$ and for all $a \in P^{\mathfrak{A}}_n$ we have $\mathfrak{M}_a \in S_{n-1}$, and $\mathfrak{A} \in T_n$ if and only if $\mathfrak{A} \in R_n$ and for all $a \in P^{\mathfrak{A}}_a$,
\begin{align*}
	\mathfrak{M}_a \in S_{n-1} &\text{ if } a \notin C^{\mathfrak{A}}_n,\\
	\mathfrak{M}_a \in T_{n-1} &\text{ if } a \in C^{\mathfrak{A}}_n.
\end{align*}
\includegraphics[scale = 0.8]{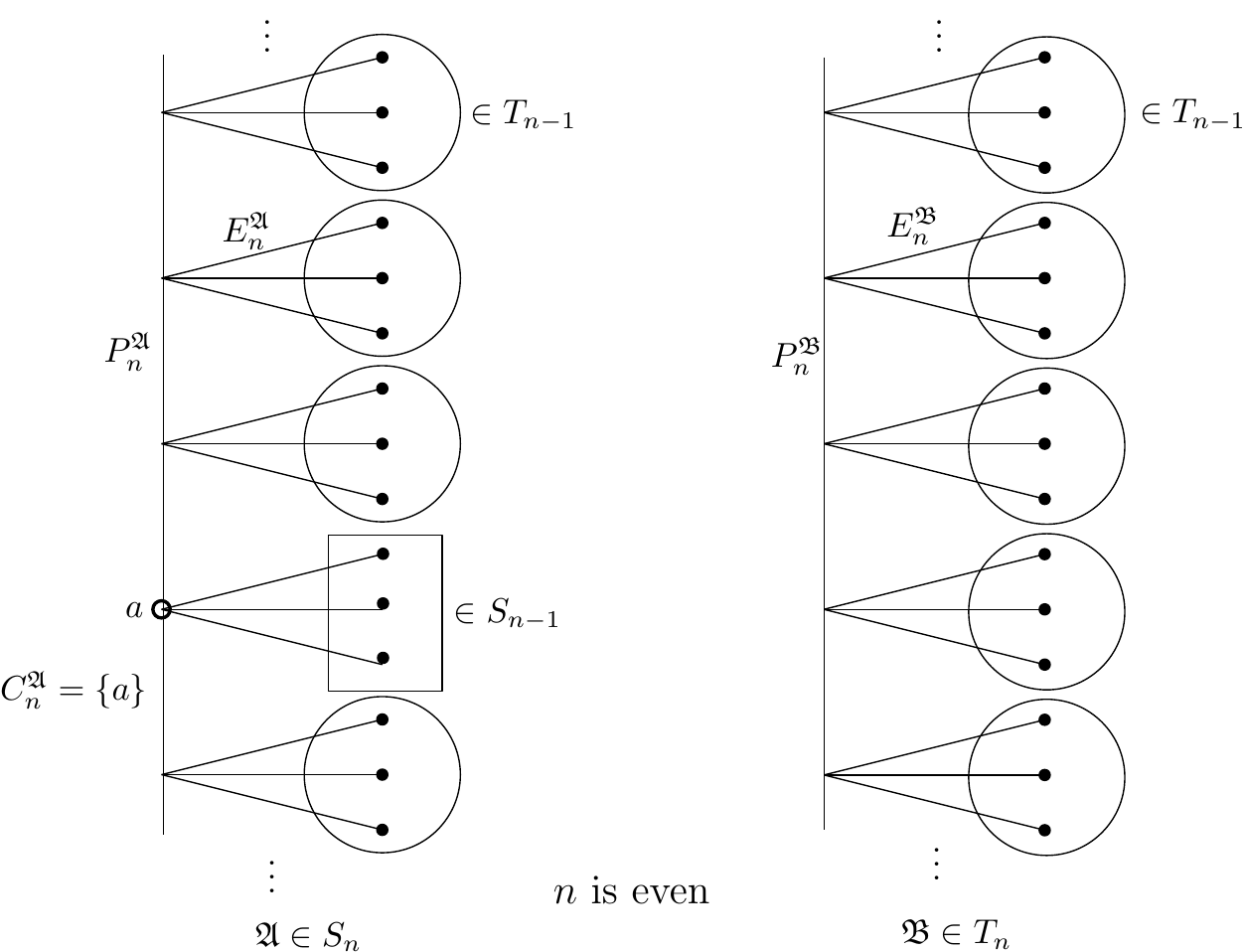}

\includegraphics[scale = 0.8]{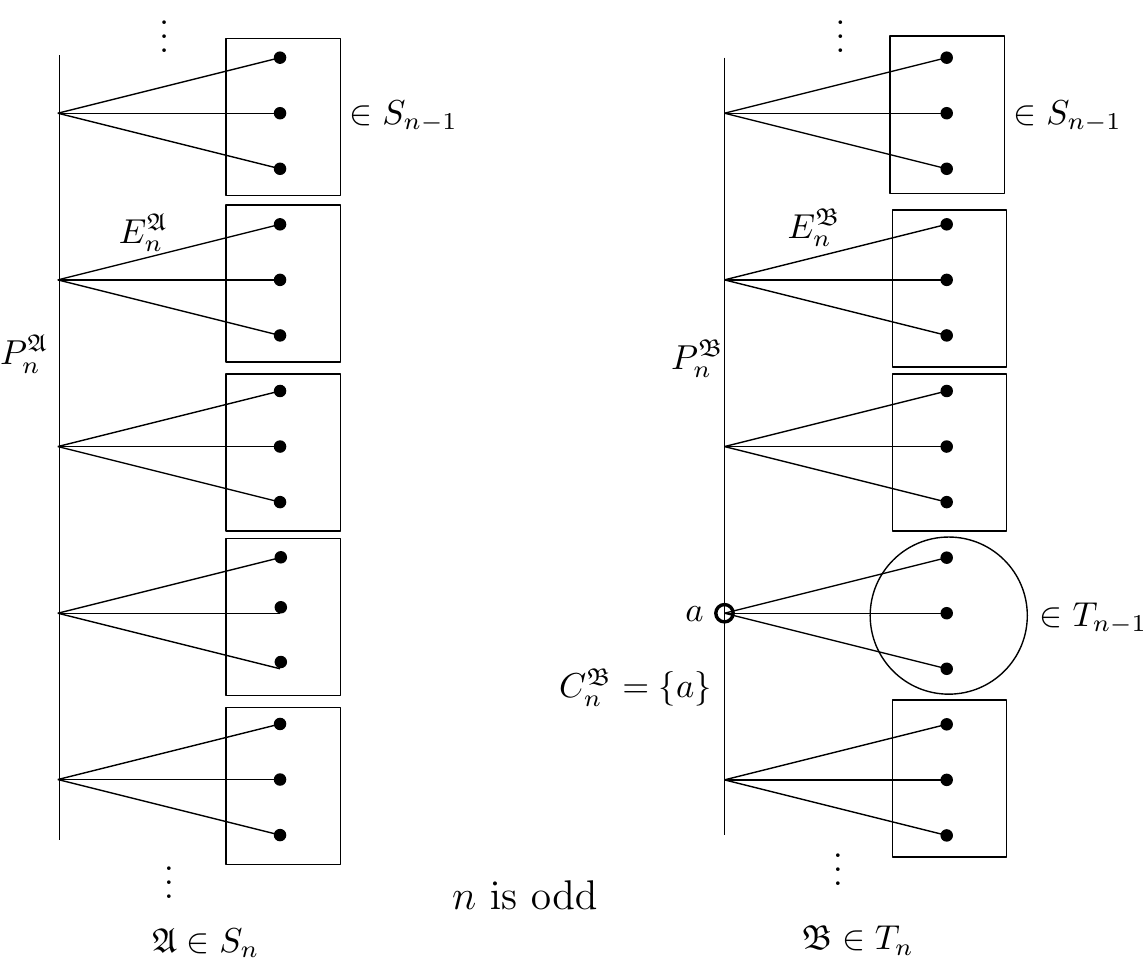}

\end{definition}

\begin{lemma}\label{exactly_one}
Let $n\geq 1$ be a natural number. Both $S_n$ and $T_n$ each have exactly one structure up to isomorphism. 
\end{lemma}

\begin{proposition}\label{suppose}
Suppose $n \geq 1$ is a natural number.
Let $\mathfrak{A} \in S_n$ and $\mathfrak{B} \in T_n$.
There is an $\mathcal{L}_{\omega\omega}[\tau_n]$-sentence $\varphi_n$ of quantifier rank $n+1$ such that $\mathfrak{A} \vDash \varphi_n$ and $\mathfrak{B} \nvDash \varphi_n$.
\end{proposition}
\begin{proof}
If $n = 1$ then we can set $\varphi_1 = \forall x (P_1(x)\rightarrow \exists y (P_0(x) \wedge E_1(x,y)))$.
Assume that $n > 1$ and such $\varphi_{n-1}$ exists.
Then if $n$ is even we can set $\varphi_n = \exists x (P_n(x) \wedge \varphi_{n-1}^{\{y : E_n(x,y)\}})$,
and if $n$ is odd we can set $\varphi_n = \forall x (P_n(x)\rightarrow \varphi_{n-1}^{\{y : E_n(x,y)\}})$.
\end{proof}

\begin{lemma}\label{natural_number}
Let $n \geq 1$ be a natural number. 
Let $\mathfrak{A}' \in S_n$, $\mathfrak{B}'\in T_n$ if $n$ is odd, and
$\mathfrak{A}'\in T_n$, $\mathfrak{B}' \in S_n$ if $n$ is even.
Set $\mathfrak{A} = \mathfrak{A}'\upharpoonright \tau_n$ and $\mathfrak{B} = \mathfrak{B}'\upharpoonright \tau_n$.
There is an embedding $f$ of $\mathfrak{A}$ into $\mathfrak{B}$ such that for any $\overline{a} \in A^{<\omega}$ there are partitions $(A_1,A_2)$ of $A$ and $(B_1,B_2)$ of $B$ such that
\begin{enumerate}
\item $\overline{a} \in A_1^{<\omega}$,
\item $f | A_1$ is an isomorphism between $\mathfrak{A} | A_1$ and $\mathfrak{B} | B_1$,
\item $\mathfrak{A}' | A_2 \cong \mathfrak{A}'$ and $\mathfrak{B}'|B_2 \cong \mathfrak{B}'$,
\item for every embedding $g :\mathfrak{A}|A_2\rightarrow \mathfrak{B}|B_2$ and $h : \mathfrak{B} |B_2\rightarrow \mathfrak{A}|A_1$ it is the case that $(f|A_1) \cup g$ and $(f^{-1}|B_1) \cup h$ are embeddings $\mathfrak{A}\rightarrow\mathfrak{B}$ and $\mathfrak{B}\rightarrow\mathfrak{A}$, respectively,
\item for any $m < \omega$, we have $(\mathfrak{A},\overline{a}) \equiv^m_{\emb} (\mathfrak{B},f\overline{a})$ if and only if $\mathfrak{A} \equiv^m_{\emb} \mathfrak{B}$.
\end{enumerate}
\end{lemma}
\begin{proof}
Recall first that for each natural number $n\geq 1$, there is an element $b_n \in B$ such that $C^{\mathfrak{B}'}_n = \{b_n\}$.
We use this element to define the ordering $\xi_n = \eta^{\mathfrak{B}}_n | \{x \in B : x  <b_n \}$.
It is easy to see by using Lemma \ref{order_isom} that for all $n\geq 1$ there is an isomorphism $h_n$ between $\eta^{\mathfrak{A}}_n$ and $\xi_n$.

We have three possible cases to consider: $n=1$, $n > 1$ is even and $n > 1$ is odd.
First suppose that $n = 1$. Then $\mathfrak{A}' \in S_1$ and $\mathfrak{B}' \in T_1$.
Define function $f : A \rightarrow B$ by setting $f(x) = h_1(x)$ if $x \in P_1^{\mathfrak{A}}$, and $f(x) = y$ if $x \in P^{\mathfrak{A}}_0$ where $y \in P^{\mathfrak{B}}_0$ is such that there is $z \in P^{\mathfrak{A}}_1$ with $\mathfrak{A} \vDash E_1(z,x)$ and $\mathfrak{B} \vDash E_1(h_1(z),y)$.
It is easy to see that $f$ is an embedding of $\mathfrak{A}$ into $\mathfrak{B}$.
Let $a_0,\dots,a_k \in A$. We must find partitions satisfying conditions 1.-5.
For this, let $b, c \in A$ be such that 
\begin{align*}
	\mathfrak{A} \vDash &P_1(b) \wedge P_0(c) \wedge E_1(b,c) \wedge \\
	&\bigwedge_{i\leq k}\big(\big(P_1(a_i) \rightarrow a_i<b\big) \wedge \big(P_0(a_i)\rightarrow a_i < c\big)\big).
\end{align*}
Set $A_1 = \{x \in A : x < b \text{ or } x < c \}$, $B_1 = f[A_1]$, and 
denote by $A_2$, $B_2$ the complements of $A_1$, $B_1$, respectively.
It is straigthforward to verify that partititions $(A_1,A_2)$, $(B_1,B_2)$ satisfy all the five requirements. 
Notice that the fact 5. follows directly from the fact 4.
Thus, the case $n = 1$ is proved. 

Assume now that $n > 1$ is even. Then $\mathfrak{A}' \in T_n$ and $\mathfrak{B}' \in S_n$, and
we have $\mathfrak{M}_x \in T_{n-1}$ for all $x \in P^{\mathfrak{A}}_n$ and $\mathfrak{M}_x \in T_{n-1}$ for all $x \in \xi_n$ so, by Lemma \ref{exactly_one}, for each $x \in P_n^{\mathfrak{A}}$ there is an isomorphism $h_x$ between $\mathfrak{M}_x$ and $\mathfrak{M}_x$.
Then $f = h_n \cup \bigcup_{x\in P^{\mathfrak{A}}_n}h_x$ is a wanted embedding.
The case where $n>1$ is odd is proved in the same way.
\end{proof}

\begin{lemma}\label{structure}
Let $n>1$ be a natural number. 
Let $\mathfrak{A}' \in S_n$, $\mathfrak{B}'\in T_n$ if $n$ is odd, and
$\mathfrak{A}'\in T_n$, $\mathfrak{B}' \in S_n$ if $n$ is even.
Set $\mathfrak{A} = \mathfrak{A}'\upharpoonright \tau_n$ and $\mathfrak{B} = \mathfrak{B}'\upharpoonright \tau_n$.
There is an embedding $f$ of $\mathfrak{B}$ into $\mathfrak{A}$ such that for any $\overline{a} \in B^{<\omega}$ there are partitions $(A_1,A_2)$ of $A$ and $(B_1,B_2)$ of $B$ such that
\begin{enumerate}
\item $\overline{a} \in B_1^{<\omega}$,
\item $f | B_1$ is an isomorphism between $\mathfrak{B} | B_1$ and $\mathfrak{A} | A_1$,
\item $\mathfrak{A}'|A_2 \in T_{n-1}$ and $\mathfrak{B}'|B_2 \in S_{n-1}$ if $n$ is even,
\item $\mathfrak{A}'|A_2 \in S_{n-1}$ and $\mathfrak{B}'|B_2 \in T_{n-1}$ if $n$ is odd,
\item for every embedding $g :\mathfrak{A}|A_2\rightarrow \mathfrak{B}|B_2$ and $h : \mathfrak{B} |B_2\rightarrow \mathfrak{A}|A_1$ it is the case that $(f^{-1}|A_1) \cup g$ and $(f|B_1) \cup h$ are embeddings $\mathfrak{A}\rightarrow\mathfrak{B}$ and $\mathfrak{B}\rightarrow\mathfrak{A}$, respectively,
\item if $n$ is even then $(\mathfrak{A},f\overline{a}) \equiv^{n-1}_{\emb} (\mathfrak{B},\overline{a})$ if and only if $(\mathfrak{M}\upharpoonright\tau_{n-1}) \equiv^{n-1}_{\emb} (\mathfrak{N}\upharpoonright \tau_{n-1})$ for all $\mathfrak{M} \in T_{n-1}$ and $\mathfrak{N} \in S_{n-1}$,
\item if $n$ is odd then $(\mathfrak{A},f\overline{a}) \equiv^{n-1}_{\emb} (\mathfrak{B},\overline{a})$ if and only if $(\mathfrak{M}\upharpoonright \tau_{n-1}) \equiv^{n-1}_{\emb} (\mathfrak{N}\upharpoonright \tau_{n-1})$ for all $\mathfrak{M} \in S_{n-1}$ and $\mathfrak{N} \in T_{n-1}$.
\end{enumerate}
\end{lemma}
\begin{proof}
There are two possible cases: $n$ is even and $n$ is odd.
Suppose first that $n$ is even.
Then $\mathfrak{A}' \in T_n$ and $\mathfrak{B}' \in S_n$.
Let $h$ be an isomorphism between $\eta^{\mathfrak{B}}_n$ and $\eta^{\mathfrak{A}}_n$.
For each $x\in P_n^{\mathfrak{B}} \setminus C_n^{\mathfrak{B}}$, let $h_x$ be an isomorphism between $\mathfrak{M}_x$ and $\mathfrak{M}_{h(x)}$. For $a \in B$ such that $C^{\mathfrak{B}}_n = \{a\}$, let $h_a$ be an embedding of $\mathfrak{M}_a$ into $\mathfrak{M}_{h(a)}$ like in Lemma \ref{natural_number}.
Let $f = h \cup \bigcup_{x\in P^{\mathfrak{B}}_n}h_x$. Then $f$ is an embedding of $\mathfrak{B}$ into $\mathfrak{A}$.
Note that $\mathfrak{M}_a \in S_{n-1}$ and $\mathfrak{M}_{h(a)} \in T_{n-1}$, and for all embeddings $g : \mathfrak{M}_a \rightarrow \mathfrak{M}_{h(a)}$ and 
$g' : \mathfrak{M}_{h(a)} \rightarrow \mathfrak{M}_a$ the functions $f' \cup g$ and $f'^{-1} \cup g'$ are also embeddings where $f' = h\cup \bigcup_{x\in P_n^{\mathfrak{B}}\setminus \{a\}}$.
Let $\overline{a} \in B^{<\omega}$, $\overline{a} = (a_1,\dots,a_k)$, and assume without loss of generality that $a_1,\dots,a_m \in B\setminus M_a$ and $a_{m+1},\dots,a_k \in M_a$ for some $m\leq k$.
For all $x \in P^{\mathfrak{A}}_n \cup P^{\mathfrak{B}}_n$, let $\mathfrak{M}'_x = \mathfrak{M}_x\upharpoonright \tau_{n-1}$.
By Lemma \ref{natural_number}, there are partitions $(M_1,M_2)$ of $M_a$ and $(N_1,N_2)$ of $M_{h(a)}$ such that
\begin{enumerate}
\item $a_{m+1},\dots,a_k \in M_1$,
\item $h_a|M_1$ is an isomorphism between $\mathfrak{M}'_a | M_1$ and $\mathfrak{M}'_{h(a)} | N_1$,
\item $\mathfrak{M}_a | M_2 \in S_{n-1}$ and $\mathfrak{M}_{h(a)} | N_2 \in T_{n-1}$,
\item for every embedding $g : \mathfrak{M}'_a|M_2 \rightarrow \mathfrak{M}'_{h(a)}|N_2$ and $g' : \mathfrak{M}'_{h(a)}|N_2 \rightarrow \mathfrak{M}'_a | M_2$ we have that 
$(h_a | M_1) \cup g$ and $(h_a^{-1} | N_1)\cup g'$ are also embeddings $\mathfrak{M}'_a \rightarrow \mathfrak{M}'_{h(a)}$ and $\mathfrak{M}'_{h(a)} \rightarrow \mathfrak{M}'_a$, respectively.
\end{enumerate}
Now let $A_1 = A \setminus N_2$, $A_2 = N_2$, $B_1 = B \setminus M_2$ and $B_2 = M_2$.
Then $(A_1,A_2)$ and $(B_1,B_2)$ are partititions of $A$ and $B$, respectively, satisfying all the seven requirements.
The case of $n$ being odd is proved in the same way.
\end{proof}

\begin{lemma}\label{we_have}
Let $\mathfrak{A}' \in S_1$ and $\mathfrak{B}' \in T_1$. Then $\mathfrak{A} \equiv_{\emb}^1 \mathfrak{B}$ where $\mathfrak{A} = \mathfrak{A}'\upharpoonright \tau_1$ and $\mathfrak{B} = \mathfrak{B}'\upharpoonright \tau_1$.
\end{lemma}
\begin{proof}
We already proved that there is an embedding of $\mathfrak{A}$ into $\mathfrak{B}$ in Lemma \ref{natural_number}, so we need to find an embedding of $\mathfrak{B}$ into $\mathfrak{A}$.
We expand the vocabulary $\tau_1$ by setting $\tau^*_1 = \tau_1 \cup \{c_q,d_q : q \text{ is a rational number } \}$. Let $h : \eta \rightarrow \eta^{\mathfrak{A}}_1$ and $g : \eta \rightarrow \eta^{\mathfrak{B}}_1$ be isomorphisms.
Let $a \in \eta$ be such that $C^{\mathfrak{B}}_1 = \{g(a)\}$.
Let $\mathfrak{A}^*$, $\mathfrak{B}^*$ be $\tau^*$-structures such that $\mathfrak{A}^* \upharpoonright \tau_1 = \mathfrak{A}$ and $\mathfrak{B}^*\upharpoonright \tau_1 = \mathfrak{B}$, $c^{\mathfrak{A}^*}_q = h(q)$, $c^{\mathfrak{B}^*}_q = g(q)$ for all $q\in \eta$, $d_q^{\mathfrak{A}^*}$, $d_q^{\mathfrak{B}^*}$ are such that $\mathfrak{C} \vDash E_1(c_q,d_q)$ for $\mathfrak{C} \in \{\mathfrak{A}^*,\mathfrak{B}^*\}$ and $q \in \eta \setminus \{a \}$, and $d_a^{\mathfrak{B}^*} = g(a)$.
We define a function $f : B \rightarrow A$ by setting
\begin{align*}
	&\left. \begin{array}{rl}
		f(c_q^{\mathfrak{B}^*}) = c_{q+1}^{\mathfrak{A}^*}, \\
		f(d_q^{\mathfrak{B}^*}) = d_{q+1}^{\mathfrak{A}^*} \\
		\end{array} \right\} \text{ if } q > a, \\
	&\left. \begin{array}{rl}
		f(c_q^{\mathfrak{B}^*}) = c_{q-1}^{\mathfrak{A}^*}, \\
		f(d_q^{\mathfrak{B}^*}) = d_{q-1}^{\mathfrak{A}^*} \\
		\end{array} \right\} \text{ if } q < a,\\
	&\ \ f(c^{\mathfrak{B}^*}_a) = c^{\mathfrak{A}^*}_a, \\
	&\ \ f(d^{\mathfrak{B}^*}_a) = d^{\mathfrak{A}^*}_{a+\frac{1}{2}}.
\end{align*}
It is straightforward to verify that $f$ is indeed an embedding of $\mathfrak{B}$ into $\mathfrak{A}$.
\end{proof}

\begin{proposition}\label{let}
Let $n \geq 1$ be a natural number and $\mathfrak{A}'\in T_n$, $\mathfrak{B}'\in S_n$.
Then $\mathfrak{A} \equiv_{\emb}^n \mathfrak{B}$, where $\mathfrak{A} = \mathfrak{A}' \upharpoonright \tau_n$ and $\mathfrak{B} = \mathfrak{B}' \upharpoonright \tau_n$.
\end{proposition}
\begin{proof}
We use induction on $n$. The base case is proved in Lemma \ref{we_have}.
Assume that $n>1$ is even and the claim is true for $n-1$.
Let Duplicator select embeddings $f : \mathfrak{A} \rightarrow \mathfrak{B}$  as in Lemma \ref{natural_number} and $g : \mathfrak{B} \rightarrow \mathfrak{A}$  as in Lemma \ref{structure}.
Suppose first that Spoiler chose embedding $f$ and elements $a_1,\dots,a_k \in A$.
Let $(A_1,A_2)$ and $(B_1,B_2)$ be partititions of $A$ and $B$, respectively, like in the statement of Lemma \ref{natural_number}.
Then $(\mathfrak{A},\overline{a}) \equiv^{n}_{\emb} (\mathfrak{B},f\overline{a})$
by the fact 5. of Lemma \ref{natural_number} and the induction hypothesis.
Now suppose that Spoiler selected the embedding $g$ and elements $b_1,\dots,b_l \in B$.
Let $(A^*_1,A^*_2)$ and $(B^*_1,B^*_2)$ be partititions of $A$ and $B$, respectively, like in Lemma \ref{structure}. Then by its 8. fact we have $(\mathfrak{A},g\overline{b}) \equiv^{n-1}_{\emb} (\mathfrak{B},\overline{b})$. 
Thus, since also $(\mathfrak{A},\overline{a}) \equiv^{n}_{\emb} (\mathfrak{B},f\overline{a})$, we have $\mathfrak{A} \equiv^n_{\emb} \mathfrak{B}$ if $n$ is even.
The case where $n$ is odd is proved in the same way.
\end{proof}

\begin{proof}[Proof of Theorem \ref{for_each}]
Combine Propositons \ref{suppose} and \ref{let}.
The existence of $\varphi_n$ having the required form with alternating existence and universal quantifiers follows from the proof of Proposition \ref{suppose}.
\end{proof}

\begin{remark}
In the proof of Theorem \ref{for_each} we used a certain tree construction 
where new structures were built from the given structures by connecting them in a specific way,
which is normal practice in results of this kind.
A similar result concerning logics with quantifiers of bounded arity was proved in \cite{Keisler: 2011} by Keisler and Lotfallah. They used a somehow similar tree construction and the bijective game (see Remark \ref{bijective}) in their proof. 
Worth mentioning are also tree-like sums of Makowsky and Shelah in \cite{Makowski: 1979} based on the ideas of Friedman \cite{Friedman: 1973} and Gregory \cite{Gregory: 1974}.
They are used to prove some results concerning, among others, Beth definability in abstract logics.
\end{remark}


\begin{thebibliography}{99}
\bibitem{Cherlin: 1986}G. Cherlin, A. H. Lachlan, \emph{Stable finitely homogeneous structures}, Trans. Amer. Math. Soc, 296:103-135, 1985. 
\bibitem{Cherlin: 2000}G. Cherlin, U. Felgner, \emph{Homogeneous finite groups}, Journal of the London Mathematical Society, 62:784-794, 2000.
\bibitem{Dawar: 2010}A. Dawar, E. Gr\"adel, \emph{Properties of almost all graphs and generalized quantifiers}, Fundamenta Informaticae, 98:351-372, 2010.
\bibitem{Ebbinghaus: 1985} H-D. Ebbinghaus, \emph{Extended logics: the general framework}, Model-theoretic logics (J. Barwise and S. Feferman, editors), Springer-Verlag, Berlin, pp. 25-76, 1985
\bibitem{Friedman: 1973} H. Friedman, \emph{Beth's theorem in cardinality logics}, Israel J. Math, 14:205-212, 1973. 
\bibitem{Gardiner: 1976} A. Gardiner, \emph{Homogeneous graphs}, Journal of Combinatorial Theory. Series B, 20:94-102, 1976.
\bibitem{Gregory: 1974} J. Gregory, \emph{Beth definability in infinitary languages}, J. Symbolic Logic, 39:22-26, 1974.
\bibitem{Hella: 1984} L. Hella, \emph{Definability hierarchies of generalized quantifiers}, Annals of Pure and Applied Logic, 43:235-271, 1989.
\bibitem{Keisler: 2011} H. J. Keisler, W. B. Lotfallah, \emph{Rank hierarchies for generalized quantifiers}, J. Log. Comput, 21:287-306, 2011.
\bibitem{Lachlan: 1984}A. H. Lachlan, \emph{On countable stable structures which are homogeneous for a finite relational language}, Israel J. Math, 49:69-153, 1984.
\bibitem{Lachlan: 1986} A. H. Lachlan, \emph{Homogeneous structures}, Proc. of the ICM. 1986.
\bibitem{Lindstrom: 1966}P. Lindstr\"om, \emph{First order predicate logic with generalized quantifiers}. Theoria, 32:186-195, 1966.
\bibitem{Makowski: 1979} J. A. Makowski, S. Shelah, \emph{The theorems of Beth and Craig in abstract model theory, I: The abstract setting}. Transactions of the American Mathematical Society, 256:215-239, 1979. 
\bibitem{Mostowski: 1957}A. Mostowski, \emph{On a generalization of quantifiers}. Fundamenta Mathematicae, 44:12-36, 1957.
\bibitem{Saracino: 1988}D. Saracino, C. Wood, \emph{Homogeneous finite rings in characteristic $2^n$}, Annals of Pure and Applied Logic, 40:11-28, 1988.
\end{thebibliography}
\end{document}